\documentclass[12pt]{amsart}
\usepackage[margin=1in]{geometry}
\usepackage{enumitem, amsmath, amsfonts, amssymb, amsthm, graphics, graphicx}
\usepackage{booktabs} 
\usepackage{mathtools} 
\usepackage[all,cmtip]{xy}
\usepackage{subcaption}
\usepackage[usenames, dvipsnames]{xcolor}
\usepackage[colorlinks=true, citecolor=Blue, linkcolor=Blue, urlcolor=Blue]{hyperref}

\graphicspath{{figures/}}

\theoremstyle{plain}
\newtheorem{thm}{Theorem}[section]

\newtheorem{conj}[thm]{Conjecture}
\newtheorem{prop}[thm]{Proposition}
\newtheorem{lemma}[thm]{Lemma}

\theoremstyle{definition}
\newtheorem{defn}[thm]{Definition}
\newtheorem{ex}[thm]{Example}
\newtheorem{remark}[thm]{Remark}

\newcommand{\defemph}[1]{\emph{\color{Blue} #1}} 

\newcommand{\X}{\mathcal{X}}

\newcommand{\Z}{\mathbb{Z}}
\newcommand{\R}{\mathbb{R}}
\newcommand{\C}{\mathbb{C}}
\newcommand{\N}{\mathbb{N}}
\newcommand{\T}{\mathbb{T}}
\newcommand{\Q}{\mathcal{Q}}

\newcommand{\sgn}{\operatorname{sgn}}

\newcommand{\G}{G}
\newcommand{\g}{\mathfrak{g}}
\renewcommand{\a}{\alpha}
\renewcommand{\b}{\beta}
\renewcommand{\L}{\mathcal{L}}

\newcommand{\inv}{^{-1}}
\newcommand{\interior}{\operatorname{int}}
\newcommand{\Hom}{\operatorname{Hom}}
\newcommand{\M}{\mathcal{M}}

\begin{document}
\title{Conserved quantities of Q-systems from dimer integrable systems}
\author{Panupong Vichitkunakorn}
\address{Department of Mathematics, University of Illinois, Urbana, IL 61821 USA}
\address{vichitk1@illinois.edu}
\begin{abstract}
We study a discrete dynamic on weighted bipartite graphs on a torus, analogous to dimer integrable systems in Goncharov-Kenyon 2013.
The dynamic on the graph is an urban renewal together with shrinking all 2-valent vertices, while it is a cluster transformation on the weight.
The graph is not necessary obtained from an integral polygon.
We show that all Hamiltonians, partition functions of all weighted perfect matchings with a common homology class, are invariant under a move on the weighted graph.
This move coincides with a cluster mutation, analog to Y-seed mutation in dimer integrable systems.
We construct graphs for Q-systems of type A and B and show that the Hamiltonians are conserved quantities of the systems.
The conserved quantities can be written as partition functions of hard particles on a certain graph.
For type A, they Poisson commute under a nondegenerate Poisson bracket.
\end{abstract}
\maketitle


\section{Introduction}

The dimer model is an important model and has a long history in statistical mechanics \cite{Kas63}.
It is a study of perfect matchings of a bipartite graph, a collection of edges in which each vertex of the graph is incident to exactly one edge.
Cluster algebras introduced in \cite{FZ} are commutative algebras equipped with a distinguished set of generators called cluster variables.
There is a transformation called mutation which creates a new set of generators from an old one.
This transformation allows us to consider a dynamical system inside the cluster algebra.
One important feature of cluster dynamics is the Laurent property, namely any cluster variables can be expanded as a Laurent polynomial on initial cluster variables.
In many cases, this expansion can be written as a partition function of dimer configurations over a certain graph, see for examples \cite{Spe07,MP07,MS10,DF13,JMZ13}.
As a discrete dynamical system, many cluster dynamics have been shown to be Liouville-Arnold integrable \cite{FM11,GK13,FH14,HI14,GSTV12,GSTV16,Wil15,FM16}.
In \cite{GK13}, dimer models are used to construct a class of integrable systems enumerated by integral polygons.

The dimer integrable system introduced in \cite{GK13} is a continuous Liouville-Arnold integrable system whose phase space is the space of line bundles with connections on a bipartite torus graph $G$ obtained from an integral convex polygon.
Let $n$ be the number of faces of $G$.
The phase space can be combinatorially viewed as the space of weights on oriented loops of $G$ compatible with loop multiplication.
The space of oriented loops is generated by all counterclockwise loops around each face $W_i$ ($i\in\{1,\dots,n\}$) and two extra loops $Z_1,Z_2$ whose homology classes on the torus $\T$ generate $H_1(\T,\Z)\cong \Z\times\Z$.
The condition $\prod_i W_i = 1$ is the only nontrivial relation among these generators.
So the phase space is generated by the weight assignment $w_i$ ($i\in\{1,\dots,n\}$) together with $z_1,z_2$ on all the loops $W_i$ and $Z_1,Z_2$.
The only condition among these weights is $\prod_i w_i = 1$.

The phase space is equipped with a Poisson bracket defined from the intersection pairing on the twisted ribbon graph obtained from $G$.
A Y-seed \cite{FZ4} of rank $n+2$ indexed by the loop generators is assigned to the weighted graph $G$ where an entry of the exchange matrix is the intersection pairing of the generators and y-variables are their weights.

Hamiltonians are defined on the phase space as Laurent polynomials on the variables $\{w_1,\dots,w_n,z_1,z_2\}.$
For $(a,b)\in\Z\times\Z$, a Hamiltonian is written as the partition function over weighted perfect matchings on $G$ where the exponent of $z_1$ and $z_2$ in their weights is $a$ and $b$, respectively.
The Hamiltonians vanish for all but finitely many $(a,b).$
There is a move called urban renewal on the weighted graph $G$, which acts on its corresponding Y-seed as a Y-seed mutation.
This transformation is a change of coordinates on the phase space, and the Hamiltonians are invariant under the transformation.
By thinking of this change of coordinates as a map on the phase space, it becomes an evolution of a discrete dynamical system.
The evolution can also be written as a Y-seed mutation.
Various discrete dynamical systems has been studied in this framwork \cite{EFS12}.

The first goal of this paper is to rethink the discrete dimer integrable system of \cite{GK13} in cluster variable setting and extend it to bipartite torus graphs not necessary obtained from integral polygons.
We study, in Sections~\ref{sec:weighted_graph} and \ref{sec:hamiltonian}, a system associated with a general bipartite torus graph not necessary be obtained from an integral polygon.
Recall that the loop weights in \cite{GK13} are the y-variables in the associated Y-seed.
In our study, we instead associate weights that act as cluster variables in the associated cluster seed.
We show that the Hamiltonians are invariant under the evolution, see Theorem~\ref{thm:hamiltonian_conserved}.


Q-systems first appeared in an analysis of Bethe ansatz of generalized Heisenberg spin chains.
They were first introduced as sets of recurrence relations on commuting variables for the classical algebras \cite{KR87} and later generalized for exceptional algebras \cite{HKOTY99}, twisted quantum affine algebras \cite{HKOTT02} and double affine algebras \cite{Her10}.
See \cite{KNS11} for a review on the subject.

They can also be normalized and then realized as mutations on cluster variables \cite{RK08,DFK09,Wil15}.
Explicit conserved quantities for Q-systems of type $A$ have been studied in \cite{DFK10} as partition functions of hard particles on a graph and in \cite{GP16} as partition functions of weighted domino tilings on a cylinder.
In \cite{GP16} the Q-system of type $A_r$ is identified with the T-system of type $A_r\otimes \hat{A}_1$.

For simply-laced finite type and twisted affine type, conserved quantities arise by identifying the systems with the dynamics of factorization mappings on quotients of double Bruhat cells \cite{Wil15}.

The second goal of the paper is to use perfect matchings on graphs to compute conserved quantities of Q-systems associated with a finite Dynkin diagram of type $A$ and $B$ in Section~\ref{sec:A} and \ref{sec:B}, respectively.
The conserved quantities of type $A$ coincide with the partition functions obtained in \cite{DFK10} and \cite{GP16}.

The paper is organized as follows.
In Section~\ref{sec:Q_clus}, we review some definitions and results on Q-systems and cluster algebras.
In Section~\ref{sec:weighted_graph}, we introduce weighted bipartite torus graphs, weight of perfect matchings and weight of loops on the graph.
A move on weighted graphs is defined.
In Section~\ref{sec:hamiltonian}, we define Hamiltonians and show that they are invariant under the move.
In Section~\ref{sec:A} (resp. Section~\ref{sec:B}), a graph for $A_r$ (resp. $B_r$) Q-system is constructed.
The Hamiltonians are shown to be conserved quantities of the system.
They can also be written as partition functions of hard particles on a certain graph.
A nondegenerate Poisson bracket is constructed.
The conserved quantities Poisson-commute in type $A$ and are conjectured to Poisson-commute in type $B$.

Throughout the paper we denote 
$[x]_+ := \begin{cases} x, & x\geq 0, \\ 0, & x<0 \end{cases}.$
For $m,n\in \N$ where $m\leq n$, let $[m,n] := \{m,m+1,\dots,n\}.$


\subsection*{Acknowledgements}
The author would like to thank his advisors R. Kedem and P. Di Francesco for their helpful advice and comments. This work was partially supported in part by Gertrude and Morris Fine foundation, NSF grants DMS-1404988, DMS-1301636, DMS-1643027 and the Institut Henri Poincar\'e for hospitality.

\section{Q-systems and cluster algebras}\label{sec:Q_clus}

In this section we review the definition of Q-systems for simple Lie algebras and an interpretation as cluster mutations.


\subsection{Q-systems}

We use a normalized version of Q-systems studied in \cite{RK08, DFK09}, and we will use it as the definition of the \defemph{Q-systems}.

Let $\g$ be a simple Lie algebra with Cartan matrix $C$.
We denote a simple root $\a$ by its corresponding integer in $[1,r]$.
The Q-system associated with $\g$ is defined to be the following recurrence relation on a set of variables $\{ Q_{\a,k}\mid \a\in [1,r],k\in\Z \}$:
\begin{align}\label{eq:q}
  Q_{\a,k+1} Q_{\a,k-1} = Q_{\a,k}^2 + \prod_{\b\sim\a}\prod_{i=0}^{|C_{\a,\b}|-1} Q_{\b,\left\lfloor\frac{t_\b k+i}{t_\a}\right\rfloor},
\end{align}
where $\lfloor t \rfloor$ denotes the integer part of $t$, and $t_\a$ are the integers which symmetrize the Cartan matrix.
That is, $t_r=2$ for $B_r$, $t_\a = 2\, (\a<r)$ for $C_r$, $t_3=t_4=2$ for $F_4$, $t_2=3$ for $G_2$, and $t_\a=1$ otherwise.

The recursions \eqref{eq:q} for type $A$ and $B$ read:
\begin{align} \label{eq:Ar_qsys}
A_r: \quad Q_{\a,k+1}Q_{\a,k-1} = Q_{\a,k}^2+Q_{\a+1,k}Q_{\a-1,k} \quad (\a=1,\dots,r),
\end{align}
\begin{align} \label{eq:Br_qsys}
B_r:\quad \left\{ \begin{aligned}
Q_{\a,k+1}Q_{\a,k-1}& = Q_{\a,k}^2 + Q_{\a+1,k}Q_{\a-1,k} \quad (\a=1,\dots,r-2),\\
Q_{r-1,k+1}Q_{r-1,k-1} &= Q_{r-1,k}^2 + Q_{r-2,k}Q_{r,2k},\\
Q_{r,k+1}Q_{r,k-1} &= Q_{r,k}^2 + Q_{r-1,\left\lfloor\frac{k}{2}\right\rfloor}Q_{r-1,\left\lfloor\frac{k+1}{2}\right\rfloor},
\end{aligned}\right.
\end{align}
with boundary conditions $Q_{0,k}=Q_{r+1,k}=1$ for $k\in\Z$.

Given a valid set of initial values $\{Q_{\a,0} = q_{\a_,0}, Q_{\a_,1} = q_{\a,1} \mid \a\in [1,r]\}$ for $q_{\a_,0},q_{\a,1}\in\C^*$, we can solve for $Q_{\a,k}$ which satisfies the Q-system for any $\a\in[1,r]$ and $k\in\Z$ in terms of the initial values.
So Q-systems can be interpreted as discrete dynamical systems where the phase space is $\X=(\C^*)^{2r}$ and the (forward) evolution is
\begin{align*} (q_{1,k}&,\dots,q_{r-1,k},q_{r,t_r k},q_{1,k+1},\dots,q_{r-1,k+1},q_{r,t_r k+1}) \\
&\mapsto (q_{1,k+1},\dots,q_{r-1,k+1},q_{r,t_r k+t_r },q_{1,k+2},\dots,q_{r-1,k+2},q_{r,t_r k+t_r+1}) 
\end{align*}
where $(q_{1,0},\dots,q_{r,0},q_{1,1},\dots,q_{r,1})$ is the initial state.

A \defemph{conserved quantity} of a discrete dynamical system is a function $H:\X\rightarrow \C$ on the phase space $\X$ which is invariant under the evolution of the system.
That is
\[ H(x_{1,k},\dots,x_{n,k}) = H(x_{1,k+1},\dots,x_{n,k+1}) \]
for all $(x_{1,k},\dots,x_{n,k}) \in \X$.

The goal of this paper is to compute conserved quantities of the Q-systems of type A and B.


\subsection{Cluster Algebras}
We review some basic definitions in cluster algebras from \cite{FZ, FZ4} and results of \cite{RK08, DFK09} about the formulations of Q-systems as cluster mutations for simple Lie algebras.

\begin{defn}[Cluster seeds and Y-seeds] \label{def:seed}
A \defemph{cluster seed} (resp. \defemph{Y-seed}) of rank $n$ is a tuple $\Sigma = (\mathbf{A},B)$ (resp. $(\mathbf{y},B)$) consisting of:
	\begin{itemize}
	\item An $n\times n$ skew-symmetrizable matrix $B$, called an \defemph{exchange matrix}.
	\item A sequence of variables $\mathbf{A}=(A_i)_{i=1}^n$ (resp. $\mathbf{y}=(y_i)_{i=1}^n$), called \defemph{cluster variables} (resp. \defemph{y-variables}).
	\end{itemize}
\end{defn}
	
When the exchange matrix $B$ is skew-symmetric, the \defemph{quiver associated with $B$}, is defined to be a directed graph without 1-cycles or 2-cycles $\Q$ whose signed adjacency matrix is $B$.
In this case, we can also write a cluster seed (resp. Y-seed) as $(\mathbf{A},\Q)$ (resp.$(\mathbf{y},\Q)$).

The mutation at $k$ on the exchange matrix can be translated to a rule of \defemph{quiver mutation} at vertex $k$ consists of the following steps:
\begin{enumerate}
\item Reverse all the arrows incident to $k$.
\item For each pair of incoming arrow from $i$ to $k$ and outgoing arrow from $k$ to $j$, we add an arrow from $i$ to $j$.
\item Remove all resulting oriented 2-cycles one by one.
\end{enumerate}

\begin{ex} \label{ex:quiver_mutation}
Consider the following quiver with five vertices labeled by $k,k_1,\dots,k_4$.
A quiver mutation at vertex $k$ can described as follows.
\begin{center}
\includegraphics[scale=0.8]{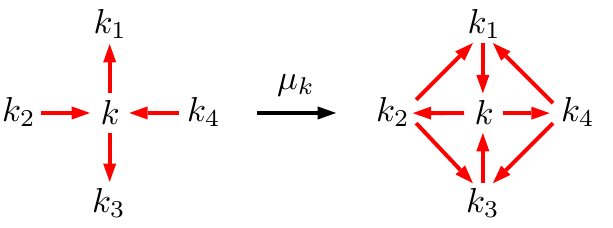}
\end{center}
\end{ex}

In this paper, we will only work on mutations at a vertex having exactly two incoming and two outgoing arrows.
So the change of arrows will be described as in Example~\ref{ex:quiver_mutation}.

\begin{defn}[Mutations and Y-seed mutations] \label{def:mutation}
For any $k\in [1,n]$, the \defemph{mutation} (resp. \defemph{Y-seed mutation}) in direction $k$ sends $(\mathbf{A},B)$ to $(\mu_k(\mathbf{A}),\mu_k(B)) = (\mathbf{A}',B')$ (resp. sends $(\mathbf{y},B)$ to $(\mu_k(\mathbf{y}),\mu_k(B))=(\mathbf{y}',B')$)  where:
	\begin{itemize}
	\item $B'_{ij} = 
		\begin{cases} 
		-B_{ij}, & i=k \text{ or } j=k,\\
		B_{ij}+\frac{1}{2}(|B_{ik}|B_{kj}+B_{ik}|B_{kj}|), & \text{otherwise}.
		\end{cases}$
	\item $A'_i = 
		\begin{cases}
		A_i, & i\neq k,\\
		A_k^{-1}\left( \prod_{B_{jk}>0}A_j^{B_{jk}} + \prod_{B_{jk}<0}A_j^{-B_{jk}} \right), & i=k.
		\end{cases}$
	\item $y'_i =
	  \begin{cases}
	  y_i^{-1}, & i = k,\\
	  y_i y_k ^{[B_{ki}]_+} (1+y_k)^{-B_{ki}}, & i \neq k.
	  \end{cases}$
	\end{itemize}
\end{defn}

\begin{defn} [$\tau-$coordinates \cite{GSV03}] \label{def:can_tran}
Let $(\mathbf{A},B)$ be a cluster seed.
There is a map $\tau$ sending $(\mathbf{A},B)$ to a Y-seed $(\mathbf{y},B)$ where
\[ y_j := \prod_{i = 1}^n A_i^{B_{ij}}\quad\text{for }j\in[1,n]. \]
\end{defn}

The map $\tau$ commutes with the mutations \cite[Proposition 3.9]{FZ4}:
\[ \tau(\mu_k(\Sigma)) = \mu_k(\tau(\Sigma)) \]
where the mutation on the left of the equation is a cluster mutation, while the mutation on the right is a Y-seed mutation.

\begin{thm} [{\cite[Theorem 3.1]{DFK09}}] \label{thm:qsys_cluster}
Let $C$ be the Cartan matrix of an underlying simple Lie algebra. The Q-system relation \eqref{eq:q} can be realized as cluster mutations. There is a sequence of mutations such that every Q-system variable appears as a cluster variable.
\end{thm}

The sequence of mutations in the theorem is explicitly described in terms of the root system of the underling Lie algebra.
We translate it into a sequence of mutation together with relabeling of indices as follows.

Let $C$ be a Cartan matrix of rank $r$, we let $\Sigma_k := (\mathbf{A}_k , B)$ be cluster seeds of rank $2r$ where
\begin{align} B = \begin{bsmallmatrix}C-C^T & C^T \\ -C & 0\end{bsmallmatrix} .\end{align}
The cluster tuple $\mathbf{A}_k$ consists of Q-system variables.
There exists a sequence of mutations $\mu$ and a permutation $\sigma \in \mathfrak{S}_{2r}$ connecting the seeds as the following.
\begin{align} \label{eq:sq_mutation}
\cdots\stackrel{\sigma\mu}{\longrightarrow} \Sigma_0 \stackrel{\sigma\mu}{\longrightarrow} \Sigma_1 \stackrel{\sigma\mu}{\longrightarrow} \Sigma_2 \stackrel{\sigma\mu}{\longrightarrow}\cdots
\end{align}

The cluster tuple $\mathbf{A}_k$, the permutation $\sigma$ and the sequence of mutations $\mu$ are defined according to the type of the Cartan matrix $C$.
The following are their definitions for type $A$ and $B$.
\begin{itemize}
\item For type $A_r$ we have  $\mu :=  \prod_{i=1}^r \mu_i$,
\begin{align} \label{eq:ArQsys_sq_mutation}
A_{i,k} := \begin{cases} Q_{i,k}, & i\in[1,r],\\ Q_{i-r,k+1}, & i\in[r+1,2r],\end{cases} \quad \sigma(i) := i+r \bmod 2r.
\end{align}
\item For type $B_r$ we have $\mu := \mu_{2r}(\prod_{i=1}^{r-1} \mu_i) \mu_r,$
\begin{align} \label{eq:BrQsys_sq_mutation}
\quad A_{i,k} := \begin{cases} Q_{i,k}, & i\in[1,r-1],\\ Q_{r,2k}, & i = r, \\ Q_{i-r,k+1}, & i\in[r+1,2r-1], \\ Q_{r,2k+1}, & i=2r,\end{cases} \quad \sigma(i): = \begin{cases} i, & i = r\text{ or }2r, \\ i+r \bmod 2r, & \text{otherwise}. \end{cases}
\end{align}
\end{itemize}
The quivers associated to the matrices $B$ for type $A$ and $B$ are shown in Figure~\ref{fig:quiver_ArBr}.
We also note that the mutations $\mu_i$ in the product $\prod_i \mu_i$ in equations \eqref{eq:ArQsys_sq_mutation} and \eqref{eq:BrQsys_sq_mutation} commute, so the product makes sense.

\begin{figure}
\includegraphics[scale=0.7]{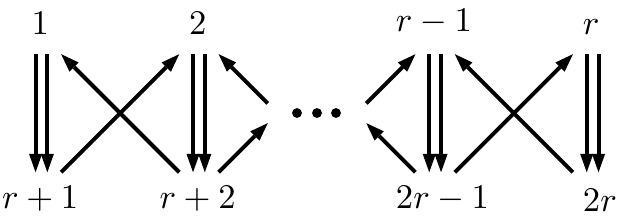}\hspace{30pt}\includegraphics[scale=0.7]{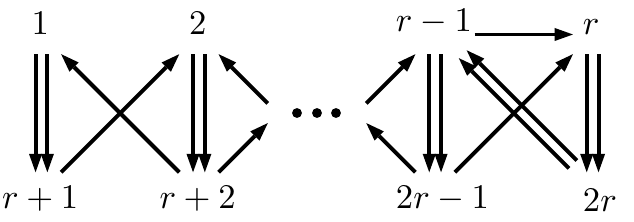}
\caption{The quivers for $A_r$ Q-system (left) and $B_r$ Q-system (right).}
\label{fig:quiver_ArBr}
\end{figure}


\section{Weighted bipartite torus graphs}
\label{sec:weighted_graph}

We think of a torus $\T = \R^2 / \Z^2$ as a rectangle with opposite edges identified.
A \defemph{torus graph} is a graph embedded on the torus with no crossing edges.
We do not require that every face is contractible.
A \defemph{weighted graph} $(\G,\mathbf{A})$ is a pair of a graph $\G$ with $n$ faces and a collection $\mathbf{A}=(A_i)_{i=1}^n$ of variables or nonzero complex numbers called \defemph{weights}.
A \defemph{bipartite graph} is a graph whose vertices can be colored into two colors (black and white) such that every edge connects two vertices of different colors.

Throughout the paper, we let $\G$ be a bipartite torus graph with $n$ faces.
We label the faces by the numbers $1$ to $n$, so $F(\G) = [1,n]$.


\subsection{Quivers associated with graphs and mutations}
\label{subsec:quiver_from_graph}

For a bipartite torus graph $\G$, we let \defemph{$\Q_\G$} be the \defemph{quiver associated with $\G$} defined as follows. 
The nodes of $\Q_\G$ are indexed by the faces of $\G$. 
There will be an arrow between node $i$ and node $j$ for each edge adjacent to face $i$ and face $j$.
The arrow is oriented in the way that the black vertex of $\G$ is on the right of the arrow, see the following figure.
\begin{center}
\includegraphics[scale=1]{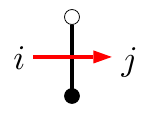}
\end{center}
Lastly, any 2-cycles are removed one by one until the directed graph has no 2-cycles.
We also let \defemph{$B_\G$} denote the signed adjacency matrix of $\Q_\G$.
Note that from the construction the resulting directed graph can possibly contain 1-cycles.

To a weighted bipartite torus graph $(\G,\mathbf{A})$ whose $\Q_\G$ has no 1-cycles, we can associate a cluster seed $(\mathbf{A},\Q_\G)$ of rank $n$, see Definition~\ref{def:seed}.

We then define two moves on weighted bipartite graphs.

\begin{defn} \label{def:urban}
An \defemph{urban renewal} \cite{Ciu98}  (a.k.a. spider move \cite{GK13}, square move \cite{Postnikov06}) at a quadrilateral face $k$ whose four sides are distinct sends $(\G,\mathbf{A})$ to $(\G',\mathbf{A}')$ as follows.
\begin{itemize}
	\item The graph $\G'$ is obtained from $\G$ by replacing the subgraph of $\G$ containing four edges around the face $k$ with a graph described in the following picture.
	The labels are face indices.
	The four outer vertices connect to the rest of the graph.
	\begin{center}
	\includegraphics[scale=0.8]{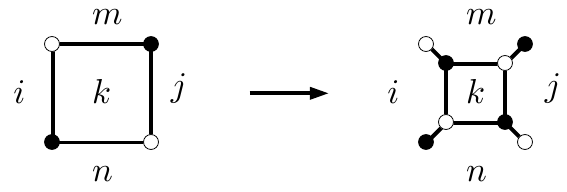}
	\end{center}
	\item The weight $\mathbf{A}' =(A'_i) = \mu_k(\mathbf{A})$ are transformed according to the cluster transformation in direction $k$ (Definition~\ref{def:mutation}).
	That is, 
	\begin{align} \label{eq:urban_on_weight}
	A'_\ell = \begin{cases}(A_i A_j + A_m A_n)/A_k, & \ell=k,\\ A_\ell, &\ell\neq k. \end{cases}
	\end{align}
\end{itemize}
\end{defn}

\begin{defn} \label{def:shrink}
A \defemph{shrinking of a 2-valent vertex} sends $(\G,\mathbf{A})$ to $(\G',\mathbf{A})$ where $\G'$ is obtained from $\G$ by removing a 2-valent vertex and identifies its two adjacent vertices, while the weights of the graph stay unchanged.
It can be visualized in the following picture.
(We also have another version of the move when the colors of the vertices are switched.)
\begin{center}
\includegraphics[scale=1]{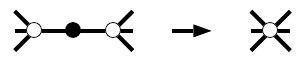}
\end{center}
\end{defn}

\begin{defn} [Mutations of weighted bipartite graphs] \label{def:graph_mutation}
Let $k$ be a quadrilateral face of a weighted bipartite graph $(\G,\mathbf{A})$.
A \defemph{mutation} at face $k$, $\mu_k$, is a combination of an urban renewal at face $k$ and shrinking of all 2-valent vertices.
Two weighted graphs are \defemph{mutation equivalent} if one is obtained from another by a sequence of mutations.
\end{defn}

\begin{thm} \label{thm:mutation_mutation}
Let $(\G,\mathbf{A})$ be a weighted bipartite graph whose $\Q_\G$ has no 1-cycles, $k$ be a quadrilateral face of $\G$.


If every pair of adjacent faces of $k$ are distinct except possibly pairs of opposite faces, then the mutation $\mu_k$ on $(\G,\mathbf{A})$ is equivalent to a cluster mutation $\mu_k$ on a seed $(\mathbf{A}, \Q_\G)$.
In particular, $\Q_{\mu_k(\G)} = \mu_k(\Q_{\G})$ and the weight on $\G$ are transformed according to the cluster transformation with respect to the quiver $\Q_\G$.
\end{thm}

\begin{proof}
Consider the following picture where $k_i$ are adjacent faces of $k$.
\begin{center}
\includegraphics[scale=.8]{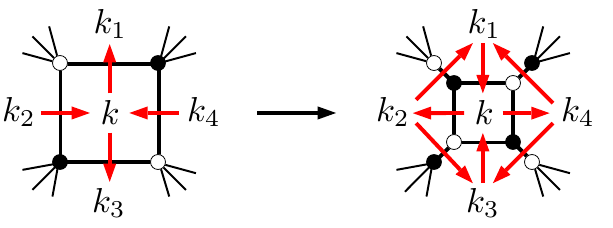}
\end{center}
Since $k_i \neq k_{i+1}$ for all $i$ (reading modulo 4), the vertex $k$ of $\Q_\G$ has exactly two incoming and two outgoing arrows.
(There is no cancellation of arrows incident to $k$.)
The arrows of $\Q_\G$ are then transformed according to the rule of quiver mutations.
Also, the mutation does not introduce any 1-cycles because $k_i$ and $k_{i+1}$ are distinct.
This also shows that the weights on the graph are transformed according the cluster mutation:
\[ A_k' = (A_{k_1}A_{k_3} + A_{k_2}A_{k_4})/ A_k. \]

Shrinking of a 2-valent vertex corresponds to removing a pair of arrows with opposite orientations that may appear between $k_i$ and $k_{i+1}$.
This also follows the rule of quiver mutations.
\begin{center}
\includegraphics[scale=1]{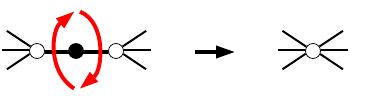}
\end{center}
So we can conclude that $\Q_{\mu_k(\G)} = \mu_k(\Q_{\G}).$
\end{proof}

\begin{ex}
The following is an example of graph mutations.
\begin{center}
\includegraphics[scale=0.7]{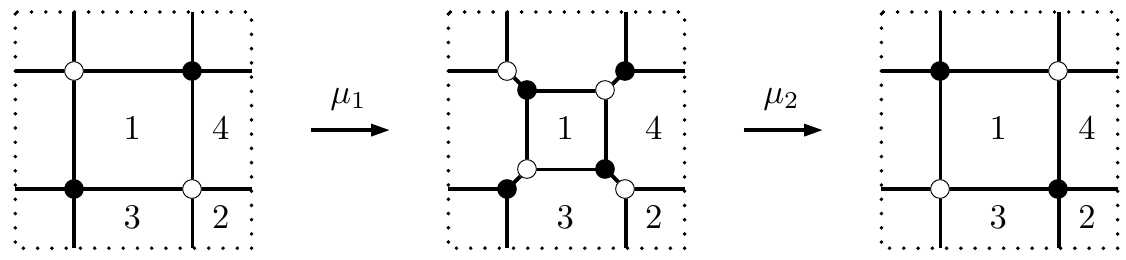}
\end{center}
We consider the first mutation $\mu_1$.
Using the notations in the proof of Theorem~\ref{thm:mutation_mutation}, when $k = 1$, we have $(k_1,k_2,k_3,k_4) = (3,4,3,4)$.
We see that $k_i \neq k_{i+1}$ for all $i\in[1,4]$.
So the mutation $\mu_1$ on the weighted graph is equivalent to the mutation $\mu_1$ on the corresponding cluster seed.
Similarly, we can see that $\mu_2$ satisfies the requirement in Theorem~\ref{thm:mutation_mutation}.

The corresponding mutations of quivers are shown as follows.
\begin{center}
\includegraphics[scale=0.7]{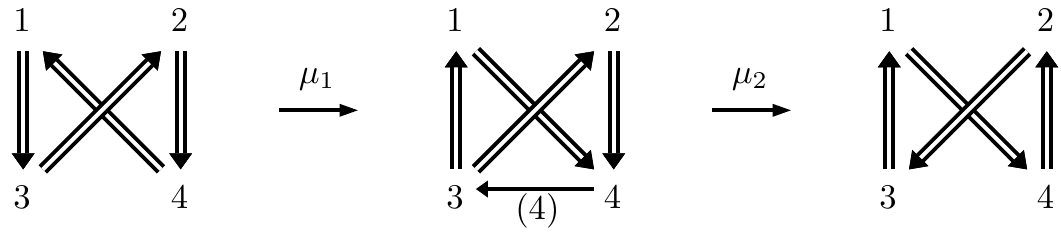}
\end{center}

The weight changes as follows.
\[ (A_1,A_2,A_3,A_4) \stackrel{\mu_1}{\longrightarrow} (A'_1,A_2,A_3,A_4)\stackrel{\mu_2}{\longrightarrow} (A'_1,A'_2,A_3,A_4) \]
where $A'_1 = (A_4^2+A_3^2)/A_1$ and $A'_2 = (A_3^2+A_4^2)/A_2$.

\end{ex}


\subsection{Perfect matchings and oriented loops}

A \defemph{perfect matching} of a graph $\G$ is a subset $M\subseteq E(\G)$ of the set of all edges in which every vertex of $\G$ is incident to exactly one edge in $M$.
\begin{defn} [Weight of perfect matchings] \label{def:weight_matching}
Let $M$ be a perfect matching of a weighted graph $(\G,\mathbf{A})$.
The \defemph{weight} of $M$ is defined by
\begin{align*}
w(M) := \prod_{e\in M}w(e).
\end{align*}
The contribution $w(e)$ is defined by $w(e) := (A_i A_j)^{-1}$ when edge $e$ is adjacent to the faces $i$ and $j$.
\end{defn}

\begin{defn} \label{def:loop_from_matching}
Let $M,M'$ be perfect matchings of $\G$.
Let \defemph{$[M]$} denote the collection of all edges in $M$ oriented from black to white.
Similarly, \defemph{$-[M]$} is the collection of edges in $M$ oriented from white to black.

We then define \defemph{$[M]-[M']$} to be the superimposition of $[M]$ and $-[M']$ with all double edges having opposite orientations removed.
In other word, $[M]-[M']$ is the set of edges $(M\setminus M')\cup (M'\setminus M)$ where an edge is oriented from black to white (resp. white to black) if it is in $M$ (resp. $M'$).
\end{defn}

It will be proved in Proposition~\ref{prop:matching_loop} that $[M]-[M']$ is a loop on $\G$ (possibly contains more than one connected component).

Let $L_1,\dots,L_m$ be loops on $\G$.
The \defemph{product of loops} $\prod_{i=1}^m L_i$ is the superimposition of $L_1,\dots,L_m$ with edges having opposite orientation removed one by one.
It is clear that the homology class of the product of loops is the sum of their homology classes.

\begin{prop} \label{prop:matching_loop}
Let $M$ and $M'$ be perfect matchings of a bipartite torus graph $\G$.
Then $[M]-[M']$ is a product of non-intersecting simple loops on $\G$.
\end{prop}
\begin{proof}
Since $M$ and $M'$ are both perfect matchings of $\G$, any vertex $v$ in $\G$ is incident to exactly one edge in $M$ and exactly one edge in $M'$ (possibly distinct edges).
If the two edges are the same, $v$ has no incident edge in $[M]-[M']$.
If the two edges are different, $v$ has exactly two incident edges in $[M]-[M']$ with different orientation. (One is oriented from black to white, while the other is from white to black.)
Hence $[M]-[M']$ is a product of non-intersecting oriented simple loops on $\G$. 
\end{proof}

Since $\G$ is a torus graph, a loop on $\G$ can be embedded on the torus.
Using the identification $H_1(\T,\Z)\cong \Z\times\Z$, we let a horizontal loop going from left to right have homology class $(1,0)$, and let a vertical loop going from bottom to top have homology class $(0,1)$.

\begin{defn} [Homology class of $M$ with respect to $M'$] \label{def:hom_of_M}
Let $M,M'$ be perfect matchings of $\G$.
Let \defemph{$[M]_{M'}$}$\in H_1(\T,\Z)$ denote the homology class of $[M]-[M']$ on the torus.
\end{defn}

\begin{defn} [Weight of paths] \label{def:weight_path}
For an oriented path $\rho$ on a weighted graph $(\G,\mathbf{A})$, we define the \defemph{weight} of $\rho$ to be
\[ w(\rho) := \prod_{e} w(e) \]
where the product runs over all directed edges in $\rho$.
Let the edge $e$ be adjacent to faces $i$ and $j$.
The contribution from $e$ is $w(e):=(A_iA_j)^{-1}$ when $e$ is oriented from black to white, while $w(e):=A_iA_j$ when $e$ is oriented from white to black.
See the following figures.
\begin{center}
\includegraphics[scale=0.8]{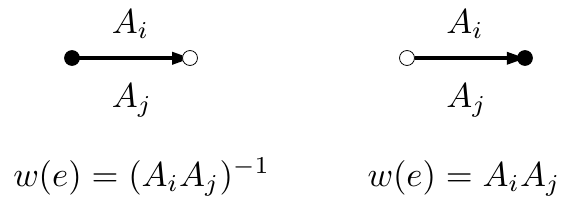}
\end{center}
\end{defn}

Comparing to the weight of perfect matchings, we notice that $$w([M]-[M']) = w(M)/w(M')$$ for any perfect matchings $M,M'$ of $\G$.

\begin{ex}
Let $\G$ be a weighted bipartite torus graph whose weight is shown at each face in the following picture.
The perfect matchings $M$ and $M'$ are depicted below.
\begin{center}
\includegraphics[scale=0.7]{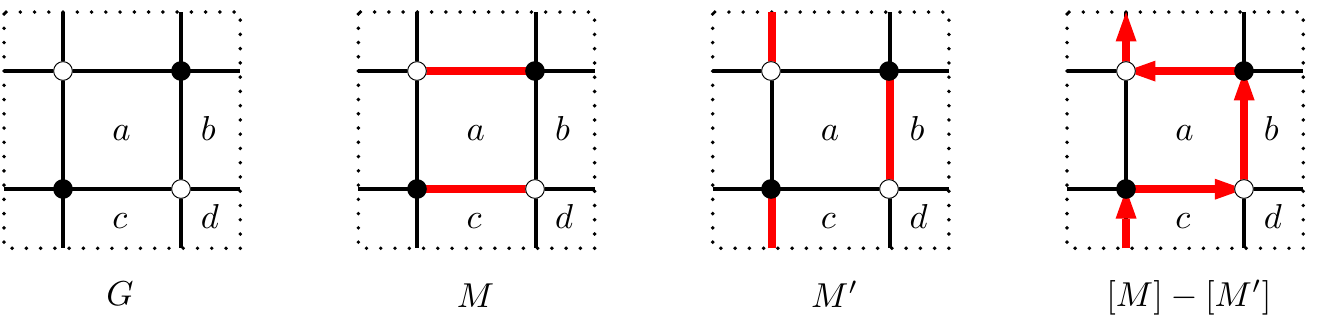}
\end{center}
We have $w(M) = a^{-2}c^{-2}$ and $w(M')= (abcd)^{-1}$.
The loop $[M]-[M]'$ has weight $w([M]-[M']) = bd/ac = w(M)/w(M')$.
Lastly, $[M]_{M'}=(0,1)$.
\end{ex}


\section{Hamiltonians}
\label{sec:hamiltonian}

In this section, we define Hamiltonians on a weighted graph $(\G,\mathbf{A})$ with respect to a perfect matching $M_0$, and show that they are invariant under certain conditions.
They will be used to compute conserved quantities of Q-systems of type $A$ in Section~\ref{sec:A} and type $B$ in Section~\ref{sec:B}.

\subsection{Hamiltonians}

We modify the Hamiltonians defined in \cite{GK13} and use the weight of perfect matchings induced from the weight of $\G$ in Definition~\ref{def:weight_matching}.

\begin{defn} \label{def:hamiltonian}
Let $(\G,\mathbf{A})$ be a weighted torus graph, $M_0$ be a perfect matching of $\G$, $(i,j)\in H_1(\T,\Z)\cong\Z\times\Z$.
The \defemph{Hamiltonian} of $(\G,\mathbf{A})$ with respect to $(i,j)$ and $M_0$ is a Laurent polynomial in $\{A_i\}_{i=1}^n$ defined by
\begin{align} \label{eq:hamiltonian}
H_{(i,j),\G,M_0}(A_1,\dots,A_n) := \sum_{M:[M]_{M_0}=(i,j)} w(M)/w(M_0), 
\end{align}
where the sum runs over all perfect matchings $M$ of $\G$ such that the homology of $[M]-[M_0]$ is $(i,j)$.
The weight $w$ is defined as in Definition~\ref{def:weight_matching}.
We say that $M_0$ is a \defemph{reference perfect matching}.
\end{defn}

\begin{defn} [Induced perfect matching by an urban renewal]
Let $k$ be a quadrilateral face of $G$.
Let $G'$ be a graph obtained from $G$ by an urban renewal at $k$.
Let $M$ be a perfect matching of $G$ containing exactly one side of $k$.
We say that a perfect matching $M'$ of $G'$ is \defemph{induced from $M$ by an urban renewal} if $M'$ coincides with $M$ on all edges of $G$ not related to the urban renewal, and on the subgraph replaced by the urban renewal $M$ and $M'$ are related as in Figure~\ref{fig:induced_urban_shrink}.
\end{defn}

\begin{defn} [Induced perfect matching by shrinking of a 2-valent vertex]
Let $G'$ be a graph obtained from $G$ by shrinking of a 2-valent vertex.
We say that a perfect matching $M'$ of $G'$ is \defemph{induced from $M$ by shrinking of a 2-valent vertex} if $M'$ coincides with $M$ on all edges of $G$ not removed by the move.
The perfect matching $M'$ can be described in Figure~\ref{fig:induced_urban_shrink}.
\end{defn}

\begin{defn} [Induced perfect matching by a mutation]
Let $\mu_k(\G)$ is a graph obtained from $\G$ by a mutation at $k$.
The \defemph{induced perfect matching from $M$ by a mutation at $k$} is the perfect matching $\mu_k(M)$ induced from $M$ by an urban renewal at $k$ and then by shrinking of all 2-valent vertices.
\end{defn}

\begin{figure}
\includegraphics[scale=0.65]{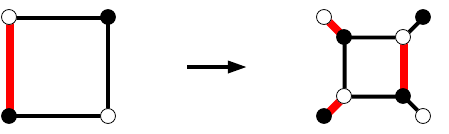} \hspace{40pt} \includegraphics[scale=.8]{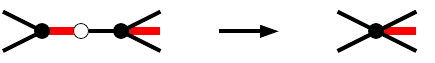}
\caption{The induced reference perfect matching from an urban renewal (left) and from shrinking of a 2-valent vertex (right).}
\label{fig:induced_urban_shrink}
\end{figure}

\begin{thm} \label{thm:hamiltonian_conserved}
Let $(\G,\mathbf{A})$ be a weighted bipartite torus graph with $n$ faces, $(\mu_k(G),\mathbf{A}')$ be obtained from $(\G,\mathbf{A})$ by a mutation at a contractible quadrilateral face $k$ of $\G$.
Let $M_0$ be a perfect matching of $\G$ containing exactly one side of $k$, 
$\mu_k(M_0)$ be induced from $M_0$ by the mutation. 
Then the Hamiltonians are invariant under the mutation:
\begin{align*}
H_{(i,j),\G,M_0}(A_1,\dots,A_n) = H_{(i,j),\mu_k(G),\mu_k(M_0)}(A_1',\dots,A_n ').
\end{align*}
\end{thm}
\begin{proof}
It suffices to show that the Hamiltonians are invariant under an urban renewal and shrinking of a 2-valent vertex.

Consider shrinking of a 2-valent vertex.
Let $\G'$ be obtained from $\G$ by shrinking a 2-valent vertex $v$.
Let assume first that the 2-valent vertex is white. 
The other case can be treated in a similar manner.
Let $v$ be adjacent to faces $i,j$ and edges $e_1,e_2$.
Consider a perfect matching $M$ of $G$.
Since $M$ is a perfect matching, exactly one of $e_1, e_2$ must be in $M$.
Assume without loss of generality that $e_1\in M$.
From the following picture, there is a unique perfect matching $M'$ of $G'$ which is identical to $M$ except on $e_1$ and $e_2$.
Similarly, any perfect matching $M'$ of $G'$ has a unique perfect matching of $G$ such that they agree on $E(G')$.
Hence this map is a bijection between the perfect matchings of $G$ and of $G'$.
\begin{center}\includegraphics[scale=0.7]{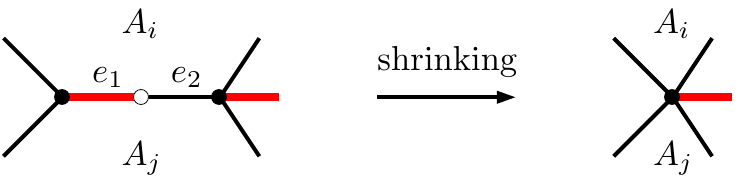}\end{center}

We also have 
\begin{align*}
w(M) = w(e_i) w(M') = (A_i A_j)\inv w(M').
\end{align*}
Similarly, for the reference perfect matching $M_0$ of $\G$, we have
\begin{align*}
w(M_0) = (A_i A_j)\inv w(M_0')
\end{align*}
regardless of whether $M_0$ contains $e_1$ or $e_2$.
Hence we have 
\begin{align*}
w(M)/w(M_0) = w(M')/w(M_0').
\end{align*}
Since the move does not change the homology class of a perfect matching  i.e. $[M]_{M_0} = [M']_{M_0'}$, from \eqref{eq:hamiltonian} we have that the Hamiltonians are invariant under the move:
\[ H_{(i,j),G,M_0}(A_1,\dots,A_n) = H_{(i,j),G',M_0'}(A_1,\dots,A_n). \]

For an urban renewal, let $(G',\mathbf{A}')$ be obtained from $(\G,\mathbf{A})$ by an urban renewal at the face $k$.
Let $M_0$ (resp. $M_0'$) be the reference perfect matching of $\G$ (resp. $G'$).
There are 4 involved edges before the move and 8 involved edges after the move.
Let $x\in\{A_i\}_{i=1}^n$ be the weight at the face $k$ of $\G$, $a,b,c,d\in\{A_i\}_{i=1}^n$ be weights at the four adjacent faces, and $x'\in \{A'_i\}_{i=1}^n$ be a weight at the face $k$ of $\G'$.
We then have
\[ x x' = ac+bd. \]
 
Without loss of generality, we assume that the edge adjacent to the faces whose weights are $x$ and $c$ is in $M_0$.
The following picture show $M_0$ and $M_0'$ on subgraph involved in the urban renewal.
\begin{align}\label{eq:toric_mutation}
\raisebox{-0.45\height}{\includegraphics[scale=0.8]{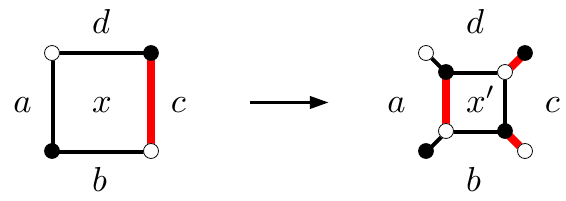}}.
\end{align}

Now we write $H_{(i,j),\G,M_0} = H_0 + H_1 + H_2$ where $H_0$ consists of all the contributions from perfect matchings containing no edges from the four sides of the face $k$, $H_1$ consists of the contributions from matchings containing one such edge, and $H_2$ consists of the contributions from matchings containing two such edges.
Similarly, we write $ H_{(i,j),G',M_0'} = H'_0 + H'_1 + H'_2$. We claim that $H_0 = H'_2$, $H_1 = H'_1$ and $H_2 = H'_0$.

In order to show $H_1 = H'_1$, we define a bijection between the matchings contributing terms to $H_1$ and the matchings contributing terms to $H'_1$.
The bijection maps a matching $M$ in $H_1$ to another matching $M'$ in $H'_1$ which differs from $M$ only on the edges involving in the urban renewal.
The bijection can be described as follows.
\begin{center}\includegraphics[scale=0.65]{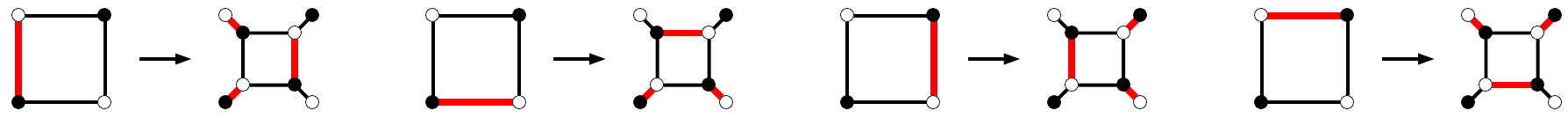}\end{center}

Since $M$ and $M'$ differ only on edges involved in the urban renewal, their contribution of the edges other than such edges near the face $k$ to $w(M)/w(M_0)$ and $w(M')/w(M_0')$ are the same.
Now we consider the contributions from the edges involved in the urban renewal.
Recalling the weights in \eqref{eq:toric_mutation}, in the first case, the contribution of such edges to $w(M)/w(M_0)$ is ${xc}/{xa}=c/a$, and the contribution to $w(M')/w(M_0')$ is $(x'abc^2d)/(x'a^2bcd) = c/a$.
Similarly, the contributions in the second case (resp. the third and the fourth case) from before and after the move are the same and equal to $c/b$ (resp. $1$ and $c/d$).
Hence $H_1 = H'_1$.

In order to show $H_0 = H'_2$, we define a one-to-two map from the matchings contributing terms to $H_0$ to the matchings contributing terms to $H'_2$.
The map can be described as the following.
\begin{center}\includegraphics[scale=0.65]{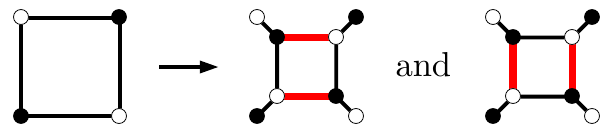}\end{center}
Let $M$ be a matching of $\G$ contributing to a term in $H_0$, $M'$ and $M''$ be the corresponding matchings of $\G'$ from the one-to-two map.
The contribution from the four edges around the face $k$ of $\G$ to $w(M)/w(M_0)$ is $xc$.
The contribution to $w(M')/w(M_0)$ (resp. $w(M'')/w(M_0)$) from the eight edges of $\G'$ resulting from the move is $x'abc^2d/(x')^2bd = ac^2/x'$ (resp. $x'abc^2d/(x')^2ac = bcd/x'$).
From \eqref{eq:toric_mutation}, $xc = ac^2/x'+bcd/x'$.
Hence $H_0 = H'_2$.

In order to show $H_2 = H'_0$, we define a two-to-one map from the matchings contributing terms to $H_2$ to the matchings contributing terms to $H'_0$ as the following.
\begin{center}\includegraphics[scale=0.65]{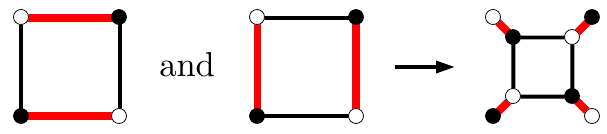}\end{center}
The contribution from the four edges around the face $k$ of $\G$ is
\[ \frac{xc}{x^2bd} + \frac{xc}{x^2ac} = \frac{c}{abcd}\frac{ac+bd}{x},\]
and the contribution from the eight edges of $\G'$ is
\[ \frac{x'abc^2d}{a^2b^2c^2d^2} = \frac{c}{abcd}x'. \]
From \eqref{eq:toric_mutation}, they are equal.
Hence $H_2 = H'_0$.
This conclude that $H_{(i,j),\G,M_0} = H_{(i,j),\G',M_0'}$.
\end{proof}

\section{\texorpdfstring{$A_r$ Q-systems}{Q-systems of type A}}
\label{sec:A}

In this section, we apply Theorem~\ref{thm:hamiltonian_conserved} to construct conserved quantities for $A_r$ Q-systems, and show that they coincide with the quantities shown in \cite{DFK10}.
We also give a Poisson structure to the phase space and show that the conserved quantities mutually Poisson-commute.


\subsection{\texorpdfstring{$A_r$ Q-systems and weighted graph mutations}{Q-systems of type A and weighted graph mutations}}

Consider the following quiver of an $A_r$ Q-system.
See Theorem~\ref{thm:qsys_cluster} for the detail.
\begin{center}
\includegraphics[scale=0.7]{quiver_Ar}
\end{center}

We recall that Theorem~\ref{thm:hamiltonian_conserved} requires every mutation to happen at a quadrilateral face.
This means every quiver mutation in the sequence $\mu$ in  \eqref{eq:ArQsys_sq_mutation} has to be at a vertex with exactly two incoming and two outgoing arrows.
In order to archive this, we add another vertex labeled by $0$ which will not be mutated, called \defemph{frozen vertex}, and assign a value of $1$ as its cluster variable.
According to the formula of a cluster mutation in Definition~\ref{def:mutation}, adding a frozen vertex with value $1$ does not effect the system.
The following is the resulting quiver $\Q$.
The vertex labeled by $0$ on the left and on the right are identified.
So it has four incident arrows.
\begin{center}
\includegraphics[scale=0.7]{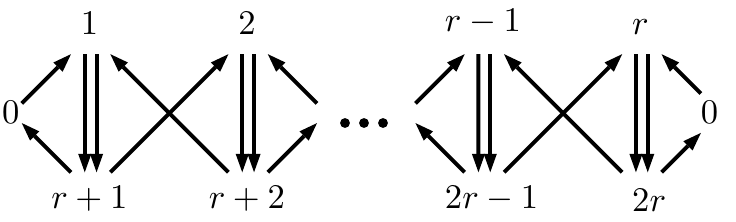}
\end{center}

The bipartite torus graph $\G$ associated with $\Q$ is depicted below.
Since every vertex of $\Q$ is of degree 4, every face of $\G$ has 4 sides.
We note that the face $0$ is not contractible.
It is indeed homotopy equivalent to a cylinder.
\begin{center}
\includegraphics[width=0.8\textwidth]{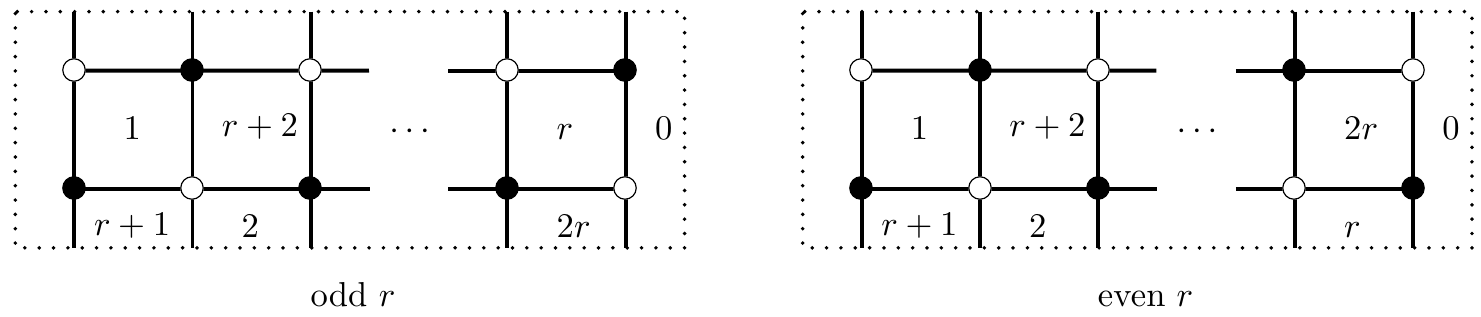}
\end{center}

Let $M_0$ be the perfect matching of $\G$ containing all vertical edges whose top vertex is black.
It can be depicted as the following.
\begin{center}
\includegraphics[scale=0.7]{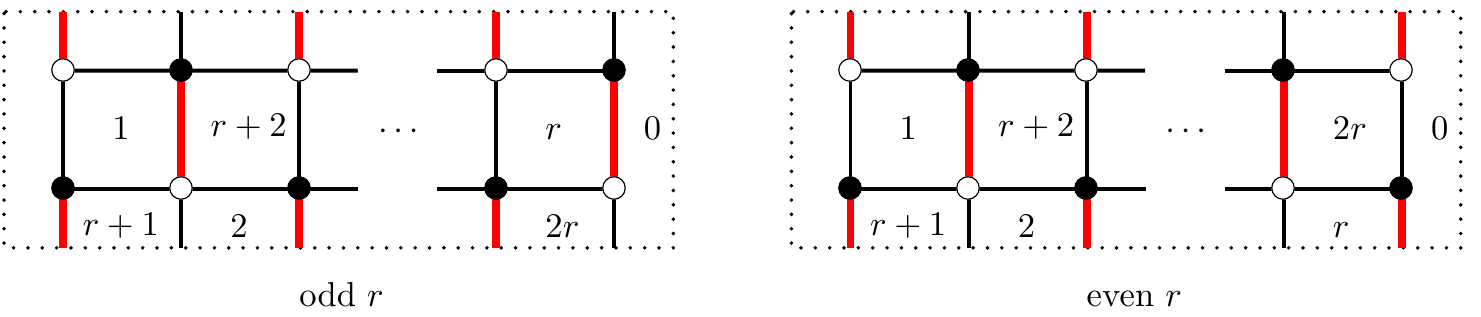}
\end{center}
We also let the weight at face $i$ be $A_i$.

\begin{thm} \label{thm:Ar_conserve}
Let $(\G,(A_{i}))$ be a weighted bipartite torus graph defined above.
Then the Hamiltonians $H_{(i,j),\G,M_0}(A_1,\dots,A_{2r})$ are conserved quantities of the $A_r$ Q-system dynamic $Q_{\a,k} \mapsto Q_{\a,k+1}$.
\end{thm}

\begin{proof}
We want to show
\begin{align*}
H_{(i,j),\G,M_0}&(Q_{1,k},\dots,Q_{r,k},Q_{1,k+1},\dots,Q_{r,k+1}) \\
& = H_{(i,j),\G,M_0}(Q_{1,k+1},\dots,Q_{r,k+1},Q_{1,k+2},\dots,Q_{r,k+2})
\end{align*}
for $k\in \Z$.

Consider $(\G,(A_{i,k}))$ where $A_{i,k}$ are defined in equation~\eqref{eq:ArQsys_sq_mutation}.
By Theorem~\ref{thm:qsys_cluster}, there is a sequence of mutation $\mu =  \mu_r\cdots\mu_1$ and a relabeling $\sigma: i \mapsto i+r \bmod 2r$ such that $\sigma\mu$ sends $\Q$ to $\Q$, and the Q-system variables are shifted by $n\rightarrow n+1$.
We first claim that the conditions in Theorem~\ref{thm:mutation_mutation} and Theorem~\ref{thm:hamiltonian_conserved} hold along the sequence of mutations $\mu$:
\begin{itemize}
\item Each mutation happens at a contractible quadrilateral face.
\item The mutating face contains exactly one edges in the induced reference perfect matching from $M_0$.
\item Two adjacent faces of the mutating face are distinct, except possibly when they are opposite faces.
\end{itemize} 

Consider the first mutation $\mu_1$.
It is clear that face $1$ is a contractible quadrilateral face and contains exactly one edge of $M_0$.
Although its adjacent faces are not distinct, the non-distinct faces are opposite.
So the conditions hold.
The following graph on the left is the resulting graph after an urban renewal at face $1$.
The graph on the right is the result after shrinking all 2-valent vertices, hence it is $\mu_1(\G)$.
The induced reference perfect matching from $M_0$ are shown on the graphs.
The weight is now $(A_{1,k+2},A_{2,k},\dots,A_{2r,k})$.
\begin{center}
\includegraphics[scale=0.7]{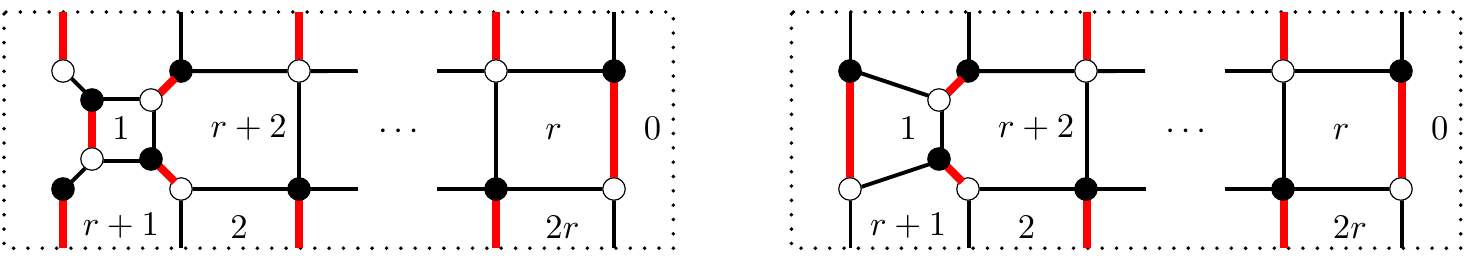}
\end{center}

We continue to the mutation $\mu_2$.
Again, we see that all the conditions are satisfied.
The following graphs are the result after an urban renewal at face $2$ and shrinking all 2-valent vertices.
The graph $\mu_2(\mu_1(G))$ is shown on the right with the weight $$(A_{1,k+2},A_{2,k+2},A_{3,k},\dots,A_{2r,k}).$$
\begin{center}
\includegraphics[scale=0.7]{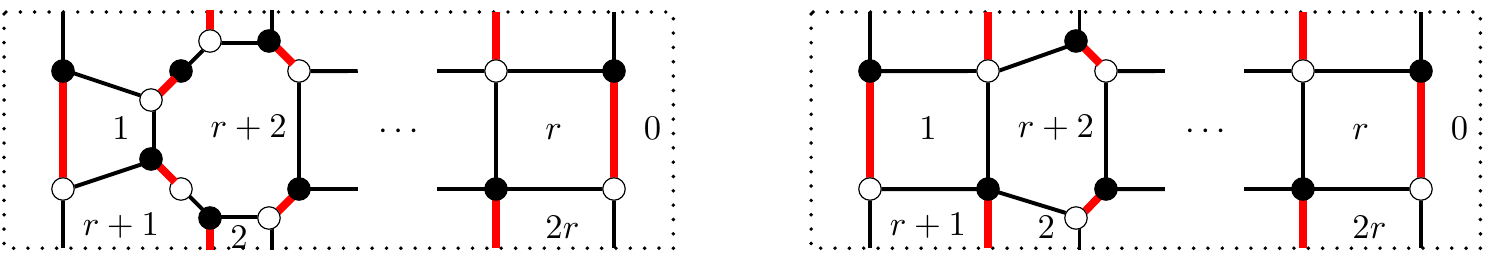}
\end{center}

Now it is easy to see that all the conditions hold at every mutation.
After applying every mutation along the sequence $\mu_1,\mu_2,\dots,\mu_r$, we have the following graph $\mu(\G)$ together with the induced reference perfect matching $\mu(M_0)$ with weight $$(A_{1,k+2},\dots,A_{r,k+2},A_{r+1,k},\dots,A_{2r,k}) = (Q_{1,k+2},\dots,Q_{r,k+2},Q_{1,k+1},\dots,Q_{r,k+1}).$$
\begin{center}
\includegraphics[scale=0.7]{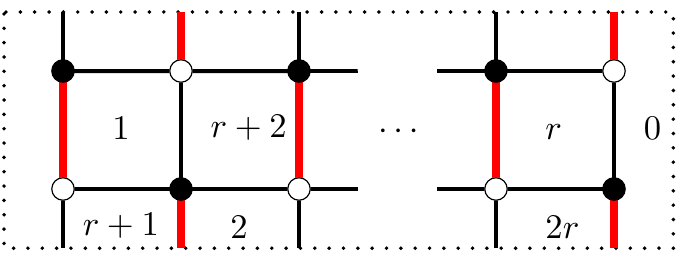}
\end{center} 

We then continue on to the relabeling $\sigma: i \mapsto i+r \bmod 2r.$
This gives the following graph on the left. 
By vertical translation, we obtain the graph on the right.
The weight after relabeling is $(Q_{1,k+1},\dots,Q_{r,k+1},Q_{1,k+2},\dots,Q_{r,k+2})$.
\begin{center}
\includegraphics[scale=0.7]{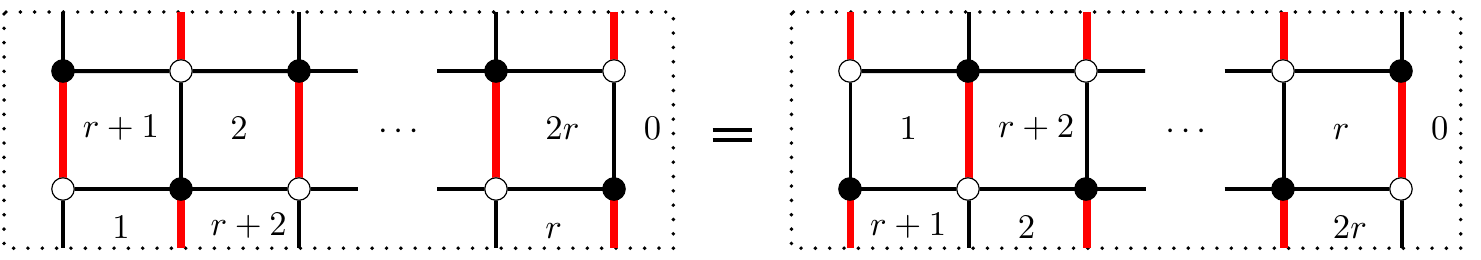}
\end{center}

We obtain back the original graph $\G$ and reference perfect matching $M_0$, while the weight is shifted from $k$ to $k+1$.
Since $,\sigma\mu(\G) = \G$ and $\sigma\mu(M_0) = M_0$, we have \[H_{(i,j),\sigma\mu(\G),\sigma\mu(M_0)} = H_{(i,j),\G,M_0}.\]
By Theorem~\ref{thm:mutation_mutation} and Theorem~\ref{thm:hamiltonian_conserved}, we get
\begin{align*}  
H_{(i,j),\G,M_0}&(Q_{1,k},\dots,Q_{r,k},Q_{1,k+1},\dots,Q_{r,k+1}) \\
&= H_{(i,j),\G,M_0}(Q_{1,k+1},\dots,Q_{r,k+1},Q_{1,k+2},\dots,Q_{r,k+2}).
\end{align*}
Hence $H_{(i,j),\G,M_0}$ are conserved quantities of the $A_r$ Q-system.
\end{proof}

\subsection{Partition function of hard particles}

In \cite{DFK10}, conserved quantities for $A_r$ Q-system are shown to be partition functions of hard particles on a certain weighted graph.
We will show that they coincide with the Hamiltonians $H_{(i,j),\G,M_0}$ computed in the previous section.

From Proposition~\ref{prop:matching_loop}, $[M]-[M_0]$ is always a product of non-intersecting simple loops of $\G$.
Given the reference perfect matching $M_0$ defined in the previous section, we then try to find such simple loops $\Gamma_i$ on $\G$.
We modify the construction in \cite{EFS12} and define $\Gamma_i$ as the following.

\begin{defn} \label{def:Ar_gamma}
For $i=1,2,\dots,2r+1$, we define $\Gamma_i$ to be the following loops.
\begin{center}
\includegraphics[width=0.8\textwidth]{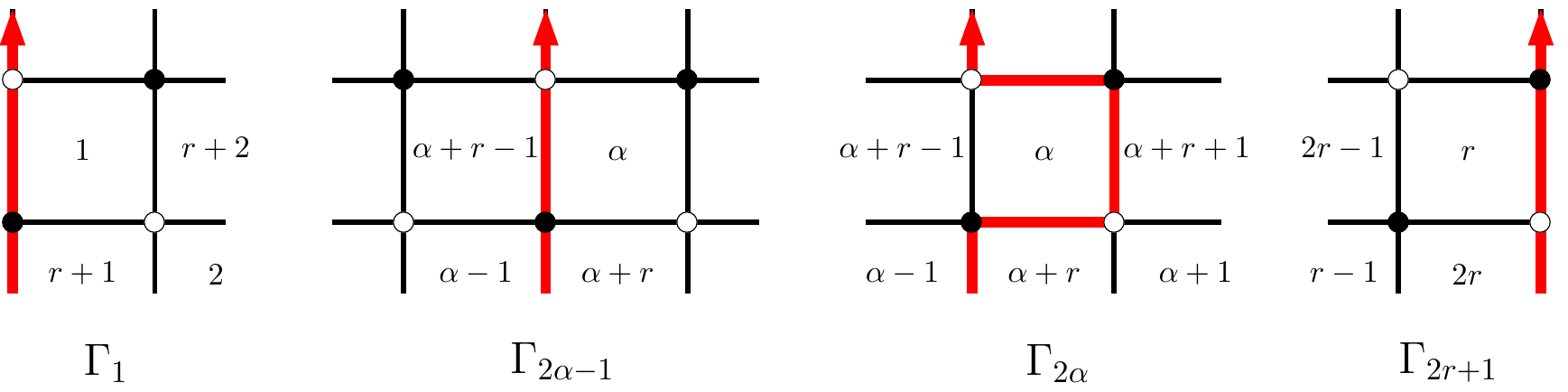}
\end{center}
\end{defn}

Let $\gamma_i := w(\Gamma_i)$ be the weight of $\Gamma_i$ on $\G$, see Definition~\ref{def:weight_path}.
We note that when $A_i = A_{i,k}$ defined in equation \eqref{eq:ArQsys_sq_mutation}, we get the same expressions as in \cite{DFK10}:
\[ \gamma_{2\alpha-1} =  \frac{Q_{\alpha-1,k}Q_{\alpha,k+1}}{Q_{\alpha,k}Q_{\alpha-1,k+1}}, \quad  \gamma_{2\alpha} = \frac{Q_{\alpha-1,k}Q_{\alpha+1,k+1}}{Q_{\alpha,k}Q_{\alpha,k+1}},  \]
where we assume that $Q_{0,k}:=1$ and $Q_{r+1,k}:=1$ for all $k\in\Z$.


\begin{thm} \label{thm:Ar_matching_gamma}
Let $M$ be a perfect matching of $\G$. 
Then $[M]-[M_0]$ is a nonintersecting collection of $\Gamma_i$'s.
Furthermore, every nonintersecting collection of $\Gamma_i$'s is $[M]-[M_0]$ for a unique perfect matching $M$ of $\G$.
\end{thm}
\begin{proof}
Consider all possible local pictures at a white vertex $v$ of $\G$.
Since $M$ (resp. $M_0$) is a perfect matching of $\G$, there is exactly one edge in $M$ (resp. $M_0$) which touches the vertex $v$.
If the two edges (one from $M$ and one from $M_0$) are the same, no loop passes through $v$.
If they are different, we have the following possibilities:
\begin{center}
\includegraphics[scale=0.6]{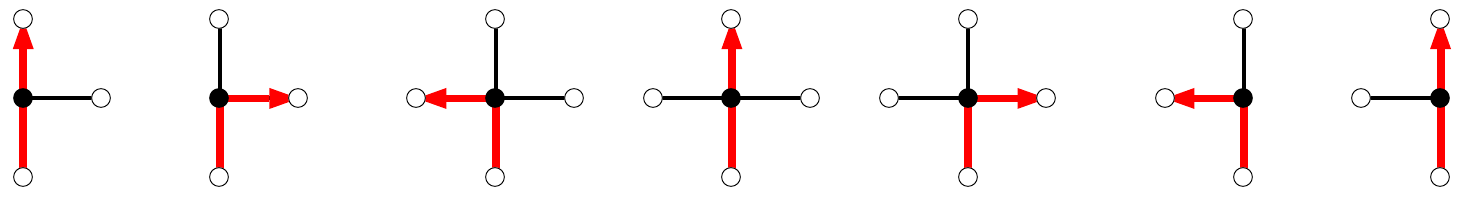}
\end{center}
Similarly, we have the following possibilities for a white vertex.
\begin{center}
\includegraphics[scale=0.6]{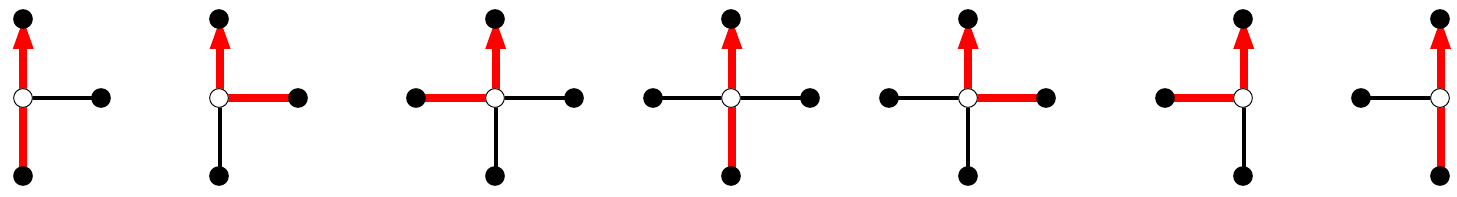}
\end{center}
From all the possibilities at black and white vertices, $[M]-[M_0]$ consists of nonintersecting loops $\Gamma_i$'s.

On the other hand, let $E$ be the edge set of a nonintersecting collection of $\Gamma_i$'s.
Notice from the pictures in Definition~\ref{def:Ar_gamma} that if there is a loop passes through a vertex $v$, one of the two incident edges belongs to $M_0$.
Then 
\[ M = (E\setminus M_0) \cup (M_0\setminus E) \]
is the unique perfect matching such that $[M]-[M_0] = \prod_{i\in I}\Gamma_i.$
\end{proof}

\begin{defn} [{\cite[Section 3.3]{DFK10}}]
Let $G_r$ be the following graph with $2r+1$ vertices indexed by the index set $[1,2r+1]$.
\begin{center}
\includegraphics[scale=0.7]{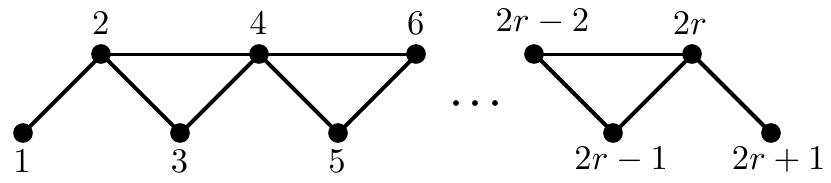}
\end{center}
\end{defn}

It is easy to see that $\Gamma_i$ intersects $\Gamma_j$ if and only if the vertices $i$ and $j$ are connected in $G_r$.
So for any subset $I\subseteq [1,2r+1]$, the loops in $\{ \Gamma_i \mid i\in I\}$ are pairwise disjoint if and only if $I$ is a subset of pairwise nonadjacent vertices of $G_r$, a.k.a. a \defemph{hard particle configuration} on $G_r$.
Hence, nonintersecting collections of $\Gamma_i$'s are parametrized by subsets of pairwise nonadjacent vertices of $G_r$.
Also, nonintersecting collections of $\Gamma_i$'s of size $n$ are in bijection with $n$-subsets of pairwise nonadjacent vertices of $G_r$.

By Theorem~\ref{thm:Ar_matching_gamma}, the possible homology classes of $[M]-[M_0]$ are $(0,k)$ for $k\in[0,r+1].$
Let 
\[ H_k := H_{(0,k),\G,M_0}. \]
It is easy to see that $H_0 = H_{r+1} = 1$.
By an interpretation of $[M]-[M_0]$ as a hard particle configuration on $G_r$, we get the following theorem and conclude that the Hamiltonians coincide with the conserved quantities computed in \cite{DFK10}.

\begin{thm} \label{thm:hamiltonian_gamma_Ar}
Let $k\in[1,r]$. Then
\[ H_{k}(A_1,\dots,A_{2r}) = \sum_{|I|=k}\prod_{i\in I} \gamma_i \]
where the sum runs over all $k$-subsets $I$ of pairwise nonadjacent vertices of $G_r$.
\end{thm}

\begin{remark}
In \cite{GP16} the Q-system of type $A_r$ is identified with the T-system of type $A_r\otimes \hat{A}_1$.
Our Hamiltonians coincide with the ``Goncharov-Kenyon Hamiltonians" in \cite[Section 3.2]{GP16}, which defined in terms of domino tilings on a cylinder.
Our choice of a reference perfect matching corresponds to the minimal-height domino tiling.
Our homology class $(0,k)$ corresponds to $k$ ``hula hoops" (noncontractible cycles) that appear in a ``double dimer model" between a domino tiling and the minimal domino tiling.
\end{remark}


\subsection{Poisson-commutation}
\label{sec:poisson_A}

Let $C$ be the Cartan matrix of type $A_r$.
The signed adjacency matrix of the quiver of $A_r$ Q-system is
\begin{align*}
B = \begin{bsmallmatrix}C-C^T & C^T \\ -C & 0\end{bsmallmatrix} = \begin{bsmallmatrix}0& C \\ -C & 0\end{bsmallmatrix}.
\end{align*}
Recall that the phase space $\X$ has coordinates $(A_1,\dots,A_{2r})$.
We then define a Poisson bracket on the algebra $\mathcal{O}(\X)$ of functions on $\X$ by
\[ \{ A_i , A_j \} = \Omega_{ij} A_i A_j  \quad (i,j\in [1,2r])\]
where the coefficient matrix $\Omega$ is defined by
\begin{align*}
\Omega = (B^T)^{-1} = -B^{-1} = \begin{bsmallmatrix} 0 & -C \\ C & 0 \end{bsmallmatrix}^{-1} = \begin{bsmallmatrix} 0 & C^{-1} \\ -C^{-1} & 0 \end{bsmallmatrix}.
\end{align*}

Using this Poisson bracket, the bracket $\{ \gamma_i, \gamma_j\}$ is in a nice form.

\begin{prop} \label{prop:bracket_gamma}
Let $i \sim j$ denote vertices $i$ and $j$ are connected in $G_r$.
Then 
\[ \{\gamma_i,\gamma_j\} = \epsilon(\Gamma_i,\Gamma_j)\gamma_i \gamma_j\] where
\begin{align*}
\epsilon(\Gamma_i,\Gamma_j) = 
	\begin{cases}
	1, & \text{if }i<j \text{ and }i \sim j,\\
	-1, & \text{if }i>j \text{ and }i \sim j. \\
	0, & \text{otherwise.}\\
	\end{cases}
\end{align*}
\end{prop}

In order to prove Proposition~\ref{prop:bracket_gamma}, we introduce an intersection pairing on a twisted ribbon graph from $\G$.
We refer to \cite{GK13} for more details.
The pairing is defined for oriented loops on $\G$.
It is skew-symmetric and can be described combinatorially as follows.

\begin{defn}
\label{def:int_pairing}
For a pair of oriented loops $W$ and $W'$ on $\G$, the \defemph{intersection pairing} $\epsilon(W,W')$ is the sum of all the following contributions over the shared vertices of $W$ and $W'$. 
\begin{center}\includegraphics[scale = 1]{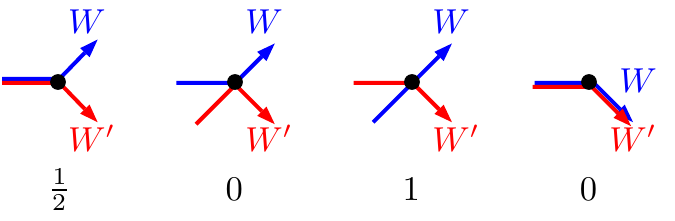}\end{center}
The sign of the contribution is switched (between plus and minus) each time the vertex coloring is switched or the orientation of a path is reversed.
\end{defn}

\begin{prop}
\label{prop:A_poisson_pairing}
For any loops $W$ and $W'$ on $\G$ with weight $w$ and $w'$, respectively.
We have
\begin{align} \label{eq:Poisson_pairing}
\{ w, w' \} = \epsilon(W,W')w w'. 
\end{align}
\end{prop}
\begin{proof}
For $j\in[1,2r]$ we let $Y_j$ be the counterclockwise loop around the face $j$ of $G$.
Since $Y_1,\dots,Y_{2r}$ together with the oriented loop $\Gamma_1$ (Definition~\ref{def:Ar_gamma}) generate all the oriented loops on $\G$, we only need to show \eqref{eq:Poisson_pairing} on such generators.
It is easy to see that
\[ \epsilon(Y_i,Y_j) = B_{ij}, \quad \epsilon(\Gamma_1, Y_j) = \delta_{j,1} - \delta_{j,r+1}. \]
Let $y_j$ be the weight of $Y_j$. We need to show that
\begin{align*}
\{ y_i, y_j \} = B_{ij} y_i y_j, \quad \{ \gamma_1 , y_j \} = (\delta_{j,1} - \delta_{j,r+1}) \gamma_1 y_j.
\end{align*}
This is equivalent to
\begin{align}
\{ \log(y_i), \log(y_j) \} &= B_{ij}, \label{eq:bracket_yy} \\
\{ \log(\gamma_1), \log(y_j) \} &= \delta_{j,1} - \delta_{j,r+1}. \label{eq:bracket_gy}
\end{align}

From the graph $\G$ we have 
\[ y_j = \prod_{i=1}^{2r} A_i^{B_{ij}},\quad \gamma_1 = \frac{A_{r+1}}{A_{1}}. \]
We also note that $\{ \log(A_i), \log(A_j) \} = \Omega_{ij}$ and $\Omega = -B\inv$. 
To show \eqref{eq:bracket_yy}, we consider
\begin{align*} \{\log( y_i), \log( y_j) \}
&= \Big\{ \sum_k B_{ki} \log(A_k) , \sum_\ell B_{\ell j}\log(A_\ell) \Big\} \\
&= \sum_{k,\ell} B_{ki}B_{\ell j} \{ \log( A_k), \log( A_\ell) \} \\
&= \sum_{k,\ell} B_{ki}B_{\ell j} \Omega_{k\ell} = \sum_{k,\ell} (-B_{ik}) (- (B\inv)_{k \ell} ) (B_{\ell j}) \\
&= (B B\inv B)_{ij} = B_{ij}.
\end{align*}
To show \eqref{eq:bracket_gy}, we consider
\begin{align*} \{ \log(\gamma_1), \log(y_j) \} 
&= \big\{ \log(A_{r+1})-\log(A_1), \sum_i B_{ij} \log(A_i) \big\} \\
&= \sum_i  \Big( \{ \log(A_{r+1}),\log(A_i) \} - \{ \log(A_1), \log(A_i) \} \Big)B_{ij} \\
&= \sum_i  (\Omega_{r+1,i} - \Omega_{1,i})B_{ij} = (\Omega B)_{r+1,j} - (\Omega B)_{1,j} \\
&= -I_{r+1,j} + I_{1,j} = -\delta_{r+1,j} + \delta_{1,j}.
\end{align*}
This proved \eqref{eq:bracket_yy} and \eqref{eq:bracket_gy}.
Hence we proved the proposition.
\end{proof}

\begin{proof}[Proof of Proposition~\ref{prop:bracket_gamma}]
It is easy to see that
\[ \epsilon(\Gamma_i,\Gamma_j) = 
	\begin{cases}
	1, & \text{if }i<j \text{ and }i \sim j,\\
	-1, & \text{if }i>j \text{ and }i \sim j. \\
	0, & \text{otherwise.}\\
	\end{cases} \] 
From Proposition~\ref{prop:A_poisson_pairing}, we have $\{\gamma_i,\gamma_j\} = \epsilon(\Gamma_i,\Gamma_j) \gamma_i\gamma_j$.
This finished the proof.
\end{proof}

Recall that a Hamiltonian can be written as
\[ H_{k} = \sum_{|I|=k}\prod_{i\in I} \gamma_i \]
where the sum runs over all $n$-subsets $I\subseteq [1,2r+1]$ of pairwise nonadjacent vertices of $G_r$.
The following lemma gives an involution which will be use to cancel out terms in a computation of $\{ H_i, H_j\}$.

\begin{lemma} \label{lem:odd_length}
Let $i_1<j_1<i_2<j_2<\dots<i_k<j_k<i_{k+1}$ be a connected sequence of vertices of $G_r$ of odd length.
Then 
\[ \{ \gamma_{i_1}\gamma_{i_2}\dots \gamma_{i_{k+1}}, \gamma_{j_1}\gamma_{j_2}\dots \gamma_{j_{k}} \} = 0. \]
\end{lemma}
\begin{proof}
We consider
\begin{align*}
 \{ \log(\gamma_{i_1}\dots \gamma_{i_{k+1}}), \log(\gamma_{j_1}\dots \gamma_{j_{k}})\} 
= \sum_{a=1}^{k+1}\sum_{b=1}^k \{ \log\gamma_{i_a}, \log\gamma_{j_b} \}  = \sum_{a=1}^{k+1}\sum_{b=1}^k \epsilon(\Gamma_{i_a},\Gamma_{j_b}).
\end{align*}
From Proposition \ref{prop:bracket_gamma}, we have
\begin{align*}
\sum_{a=1}^{k+1}\sum_{b=1}^k \epsilon(\Gamma_{i_a},\Gamma_{j_b}) 
= \sum_{a=1}^k  \epsilon(\Gamma_{i_a},\Gamma_{j_a})  + \sum_{b=1}^k  \epsilon(\Gamma_{i_{b+1}},\Gamma_{j_b}) = k - k = 0.
\end{align*}
The $2k$ contributions on the right hand side of the first equality can be view graphically as the following.
\begin{center}
\includegraphics[scale=1]{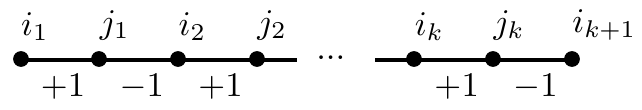}
\end{center} 
Hence $\{ \gamma_{i_1}\gamma_{i_2}\dots \gamma_{i_{k+1}}, \gamma_{j_1}\gamma_{j_2}\dots \gamma_{j_{k}} \} = 0.$
\end{proof}

\begin{thm} \label{thm:hamiltonian_commute_A}
The Hamiltonians Poisson-commute.
\end{thm}
\begin{proof}
This proof is adapted from the proof of \cite[Theorem 3.7]{GK13}.
Let $m,n\in[1,r]$.
We would like to show that $\{H_m,H_n\}=0$.

Consider an involution $\iota:(I,J)\mapsto(I',J')$ on the set of pairs of index subsets $(I,J)$ where $I,J\subseteq [1,2r+1]$ are subsets of pairwise nonadjacent vertices of $G_r$, $|I|=m$ and $|J|=n$.
For a pair $(I,J)$, we define $(I',J')$ by the following steps.

First, we think of $I$ and $J$ as subsets of $V(G_r)$, and then plot all the elements of $I$ and $J$ on the graph $G_r$.
For each even length maximal chain of vertices, we have an alternating sequence between elements of $I$ and elements of $J$.
Then $I'$ (resp. $J'$) is obtained from $I$ (resp. $J$) by swapping all the elements in every even-length maximal chain.
As a result, $|I'|=|I|$ and $|J'|=|J|$.
So, both $I$ and $I'$ (resp. $J$ and $J'$) contribute terms to $H_m$ (resp. $H_n$).
It is also clear that $\iota$ is an involution.

From Lemma~\ref{lem:odd_length}, the vertices in odd-length maximal chains contribute nothing to the bracket $\{ \prod_{i\in I}\gamma_i,\prod_{j\in J}\gamma_j \}$.
So
\[
\big\{\log( \prod_{i\in I}\gamma_i),\log(\prod_{j\in J}\gamma_j) \big\} 
= \sum_{C} \big\{ \log(\prod_{i\in C\cap I}\gamma_i),\log(\prod_{j\in C\cap J}\gamma_j) \big\} 
\]
where the sum runs over all even-length maximal chain $C\subseteq I\cup J$.
By the construction, $C\cap I' = C\cap J$ and $C\cap J'=C\cap I$.
Hence, 
\[ \big\{ \prod_{i\in I}\gamma_i,\prod_{j\in J}\gamma_j \big\} = -\big\{ \prod_{i\in I'}\gamma_i,\prod_{j\in J'}\gamma_j \big\}. \]
A fixed point of $\iota$ is a pair $(I,J)$ where all maximal chains are odd.
We have that $$\{ \prod_{i\in I}\gamma_i,\prod_{j\in J}\gamma_j \}= 0.$$
Hence, $\{ H_m,H_n\} = 0 $.
\end{proof}

\begin{ex}
Let $r=1$, $I=\{1,4\}$ and $J=\{-1,2,5\}$.
The following picture show three maximal chains of $I\cup J$: $\{ -1\}$, $\{1,2\}$ and $\{4,5\}$.
\begin{center}
\includegraphics[scale=1]{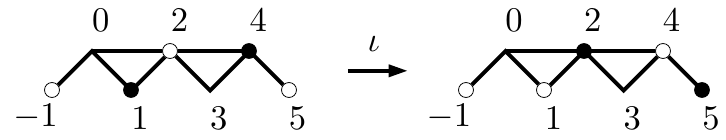}
\end{center}
So $I'=\{2,5\}$ and $J' = \{ -1,1,4\}$.
The black (resp. white) dots are elements of $I$ and $I'$ (resp. $J$ and $J'$).
\end{ex}


\subsection{Another proof of Theorem~\ref{thm:hamiltonian_commute_A}}

The proof is based on the proof of \cite[Theorem 3.7]{GK13}.
We denote by $\overline{\G}$ the following bipartite torus graph.
It differs from $\G$ by two extra edges.
\begin{center}
\includegraphics[scale=.8]{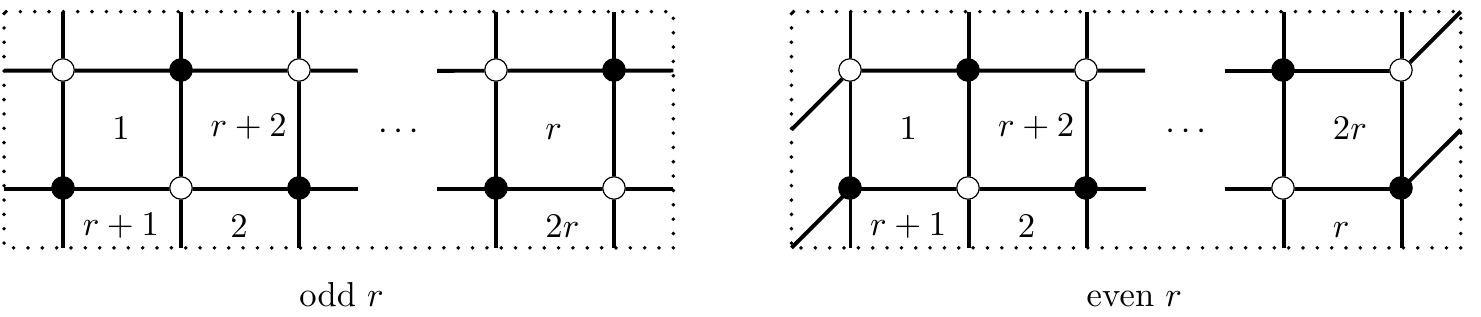}
\end{center}
Using notations in Section~\ref{sec:bipartitegraph}, it can be obtained from an integral convex polygon with edge vectors: 
\begin{itemize}
\item $e_1=(1,\frac{r+1}{2})$, $e_2=(-1,\frac{r+1}{2})$, $e_3=(-1,-\frac{r+1}{2})$, $e_4=(1,-\frac{r+1}{2})$ when $r$ is odd
\item $e_1=(1,\frac{r+2}{2})$, $e_2 = (-1,\frac{r}{2})$, $e_3=(-1,-\frac{r+2}{2})$, $e_4=(1,-\frac{r}{2})$ when $r$ is even.
\end{itemize}
The integral convex polygon can be depicted as the following.
\begin{center}
\includegraphics[scale=.7]{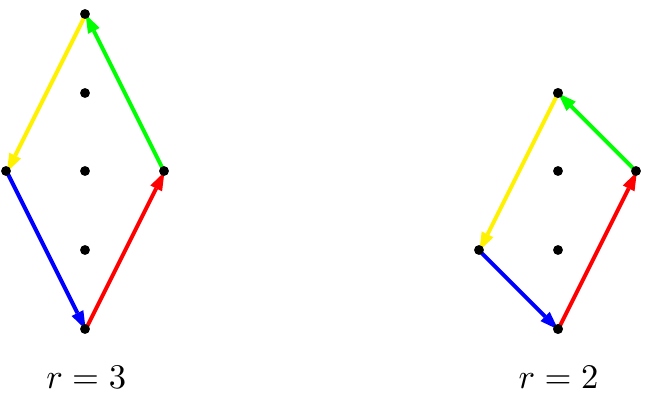}
\end{center}
The following pictures are $\overline{G}$ with oriented loops.
\begin{center}
\includegraphics[scale=.8]{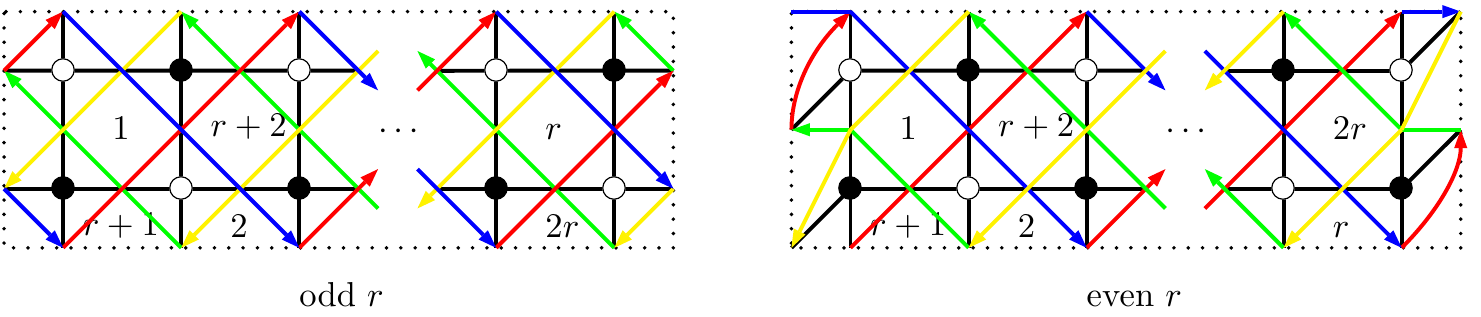}
\end{center}

We define a reference perfect matching from a sequence $e_1,e_2,e_3,e_4$ according to the construction in \ref{sec:hamiltonian_GK}.
It coincides with the perfect matching $M_0$ defined at the beginning of Section~\ref{sec:A}.
We denote by $M_0$ this perfect matching of $\overline{\G}$.

Let $\overline{\G}$ be a bipartite torus graph obtained from an integral polygon with edge vectors $e_1,e_2,\dots,e_n$ in Section~\ref{sec:hamiltonian_GK}.
Let $v$ be a vertex in a loop $\Gamma$.
Let $\varphi_{M_0} : E(\overline{\G})\rightarrow \{0,1\}$ defined by $\varphi_{M_0}(e) = 1$ if and only if $e\in M_0$.
Define $b_v(\Gamma,M_0)$ by
\[ b_v(\Gamma,M_0) := \sum_{e\in R_v} \varphi_{M_0}(e) - \sum_{e\in L_v} \varphi_{M_0}(e) \in\Z \]
where $R_v$ (resp. $L_v$) be the set of all edges incident to $v$ which are on the right (resp. left) of the loop $\Gamma$ (not including edges in $\Gamma$).

The following lemma says that for any homologically nontrivial loop $\Gamma$ on $\overline{\G}$, the number of edges in $M_0$ incident to $\Gamma$ on the left is equal to the number of edges in $M_0$ incident to $\Gamma$ on the right.

\begin{lemma} [{\cite[Lemma 3.9]{GK13}}]
\label{lem:bending}
Let $M_0$ be the reference perfect matching obtained from edge vectors. (See Section~\ref{sec:hamiltonian_GK}.)
For any simple topologically nontrivial loop $\Gamma$ on $\overline{\G}$, 
\[ \sum_{v\in \Gamma} b_v(\Gamma,\varphi_{M_0}) = 0. \]
\end{lemma}

Let $a,b\in H_1(\T,\Z)$.
We will show that $\{ H_{a,\overline{\G},M_0}, H_{b,\overline{\G},M_0} \} = 0$.
We simplify the notation by letting $H_a := H_{a,\overline{\G},M_0}$.
Let $\M_a$ be the set of all perfect matchings of $\overline{\G}$ whose homology class with respect to $M_0$ is $a$, i.e. $[M]_{M_0} = a \in\Z\times\Z$.
Then $H_a = \sum_{M\in\M_a} w(M)/w(M_0)$.

Let $M_1\in\M_a$ and $M_2\in\M_b$.
By Proposition~\ref{prop:matching_loop}, $[M_1]-[M_2]$ is a collection of non-intersecting simple loops.
There are two types of simple loops: homologically trivial loops and homologically nontrivial loops.
Define an involution $\iota:\M_a\times\M_b \rightarrow \M_a\times\M_b$ by $(M_1,M_2)\mapsto(M_1',M_2')$ where $M_1', M_2'$ are obtained from $M_1, M_2$ by exchanging all edges in each homologically trivial loop of $[M_1]-[M_2]$.
See the following example.
\begin{center}
\includegraphics[width = \textwidth]{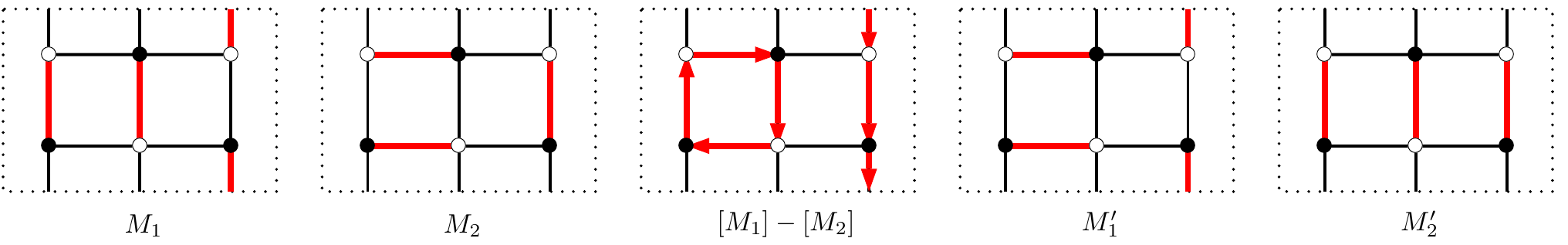}
\end{center}

The involution $\iota$ define an equivalence relation on $\M_a\times\M_b$ by $(M_1,M_2)\sim (M_1',M_2')$ if $\iota((M_1,M_2)) = (M_1',M_2')$.
So each equivalence class has either one or two elements.
Each fixed point of $\iota$ is in its own singleton class.

Consider
\begin{align*}
\{ H_a, H_b \} &= \sum_{(M_1,M_2)\in\M_a\times\M_b} \left\{ \frac{w(M_1)}{w(M_0)} , \frac{w(M_2)}{w(M_0)}  \right\}
\end{align*}
The sum is then divided into two sums: 
\begin{itemize}
\item The first summation runs over all the fixed points of $\iota$:
\begin{align} \label{eq:1sum}
\sum_{\{(M_1,M_2)\}} \left\{ \frac{w(M_1)}{w(M_0)} , \frac{w(M_2)}{w(M_0)} \right\}
\end{align}
\item The second summation runs over all equivalence classes $\left\{(M_1,M_2),(M_1',M_2')\right\}$ of size 2:
\begin{align} \label{eq:2sum} 
\sum_{ \{(M_1,M_2),(M_1',M_2')\} } \left\{ \frac{w(M_1)}{w(M_0)} , \frac{w(M_2)}{w(M_0)} \right\} + \left\{ \frac{w(M_1')}{w(M_0)} , \frac{w(M_2')}{w(M_0)} \right\}
\end{align}
\end{itemize}

To calculate the first sum \eqref{eq:1sum}, let $(M_1,M_2)$ be a fixed point of $\iota$.
By Proposition~\ref{prop:A_poisson_pairing} we have
\[ \left\{ \frac{w(M_1)}{w(M_0)} , \frac{w(M_2)}{w(M_0)} \right\} = \sum_{v} \epsilon_v \left( [M_1]-[M_0],[M_2]-[M_0] \right) \frac{w(M_1)w(M_2)}{w(M_0)^2}
\]
where $\epsilon_v$ is the contribution from vertex $v$ to the intersection pairing.
For each vertex $v$, there are edges $e_0,e_1,e_2$ incident to $v$ and belong to $M_0$, $M_1$, $M_2$, respectively.
If $e_1 = e_2$, then $\epsilon_v \left( [M_1]-[M_0],[M_2]-[M_0] \right) = 0$.
If $e_1\neq e_2$, the vertex $v$ must belong to the loop $[M_1]-[M_2]$, and vice versa.
Since $[M_1]-[M_2]$ is a fixed point of $\iota$, it has no homologically trivial loops.
So every loop is homologically nontrivial.

Consider a homologically nontrivial simple loop $\Gamma$ of the loop $[M_1]-[M_2]$ (a connected component of $[M_1]-[M_2]$).
We have the following four configurations of $e_0,e_1,e_2$ depending whether $e_0$ is on the left/right of $\Gamma$ or whether $v$ is black/white.
Note that $e_1$ (resp. $e_2$) goes from black to white (resp. white to black) in $\Gamma$.
\begin{center}
\includegraphics[scale=.8]{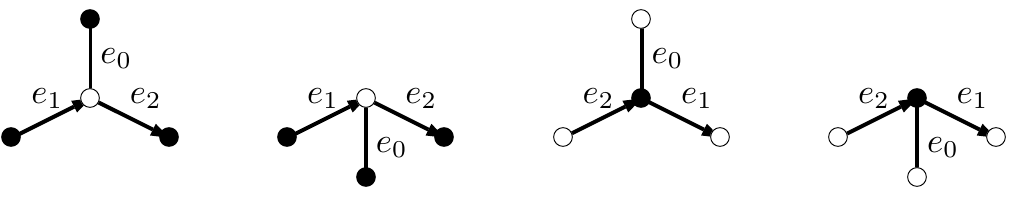}
\end{center}
The local contribution to $\epsilon_v \left( [M_1]-[M_0],[M_2]-[M_0] \right)$ from each configuration is $1/2$, $-1/2$, $1/2$ and $-1/2$, respectively.
We can conclude that the contribution is $1/2$ (resp. $-1/2$) when $e_0$ is on the left (resp. right) of $\Gamma$.
By Lemma~\ref{lem:bending}, we can conclude that $$\epsilon_v \left( [M_1]-[M_0],[M_2]-[M_0] \right) = 0.$$
Hence the first summation vanishes.

For the second sum, since $w(M_1)w(M_2) = w(M_1')w(M_2')$ the summand is equal to 
\begin{align} \label{eq:summand}
\sum_v \left( \epsilon_v \left( [M_1]-[M_0],[M_2]-[M_0] \right) + \epsilon_v \left( [M_1']-[M_0],[M_2']-[M_0] \right) \right) \frac{w(M_1)w(M_2)}{w(M_0)^2}.
\end{align}
For each vertex $v$, there are edges incident to $v$ and belong to $M_1$, $M_2$ and $M_0$.
Similarly to the computation for the first sum, if such edges of $M_1$ and $M_2$ are the same, then the summand vanishes.
If $v$ belongs to a homologically nontrivial loop, by Lemma~\ref{lem:bending}, the summand also vanishes.
If $v$ belongs to a homologically trivial loop, we see that $[M_1]-[M_0]$ (resp. $[M_2]-[M_0]$) is locally the same as $[M_2']-[M_0]$ (resp. $[M_1']-[M_0]$) at $v$.
Hence
\begin{align*} \epsilon_v \left( [M_1']-[M_0],[M_2']-[M_0] \right)
&= \epsilon_v \left( [M_2]-[M_0],[M_1]-[M_0] \right)\\
&= - \epsilon_v \left( [M_1]-[M_0],[M_2]-[M_0] \right).
\end{align*}
Hence the summand \eqref{eq:summand} vanishes.
This concludes that the second summation vanishes.

Thus $\{ H_a, H_b \} = 0$.
This finishes the proof.


\section{\texorpdfstring{$B_r$ Q-systems}{Q-systems of type B}}
\label{sec:B}

We construct a ``double cover" of $B_r$ Q-system quiver and compute Hamiltonians on the weighted bipartite torus graph associated with the quiver.


\subsection{\texorpdfstring{$B_r$ Q-systems and weighted graph mutations}{Q-systems of type B and weighted graph mutations}}

Consider the following quiver of an $B_r$ Q-system.
See Theorem~\ref{thm:qsys_cluster} for the detail.
\begin{center}
\includegraphics[scale=0.7]{quiver_Br}
\end{center}
In the case of $A_r$ Q-system, we added a frozen vertex to the quiver so that it has an associated bipartite torus graph.
Unlike the previous case, the quiver of $B_r$ Q-system does not have an associated bipartite torus graph.
We resolve this issue by considering instead a double-cover of the quiver together with a frozen vertex as in the following picture.
Let $\Q$ denote the resulting quiver.
\begin{center}
\includegraphics[scale=0.7]{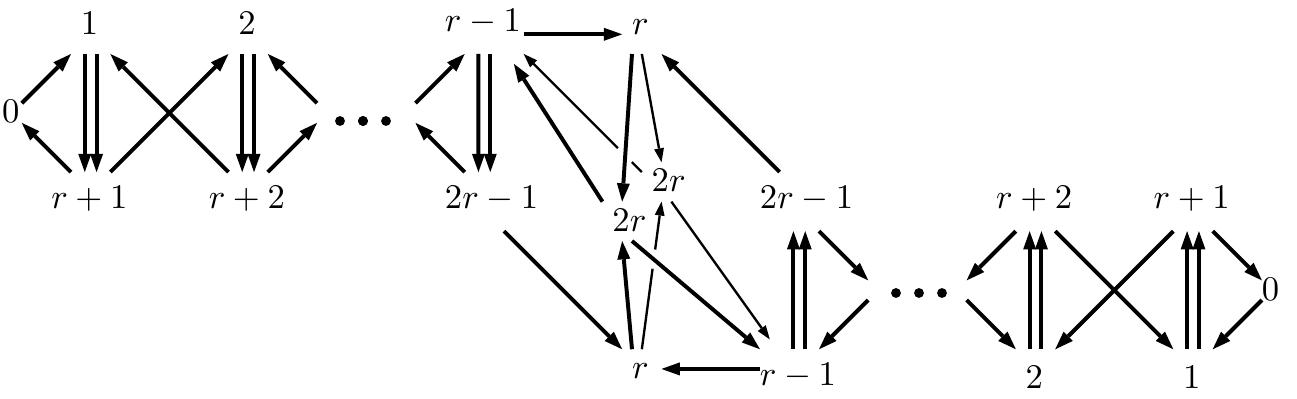}
\end{center}
Notice that $\Q$ and the original quiver are locally the same.
There are two copies for each vertex.
So we can think of this quiver similar to a double-cover of the original quiver with an extra frozen vertex.
The bipartite torus graph $\G$ associated with $\Q$ is depicted below.
\begin{center}
\includegraphics[width=0.6\textwidth]{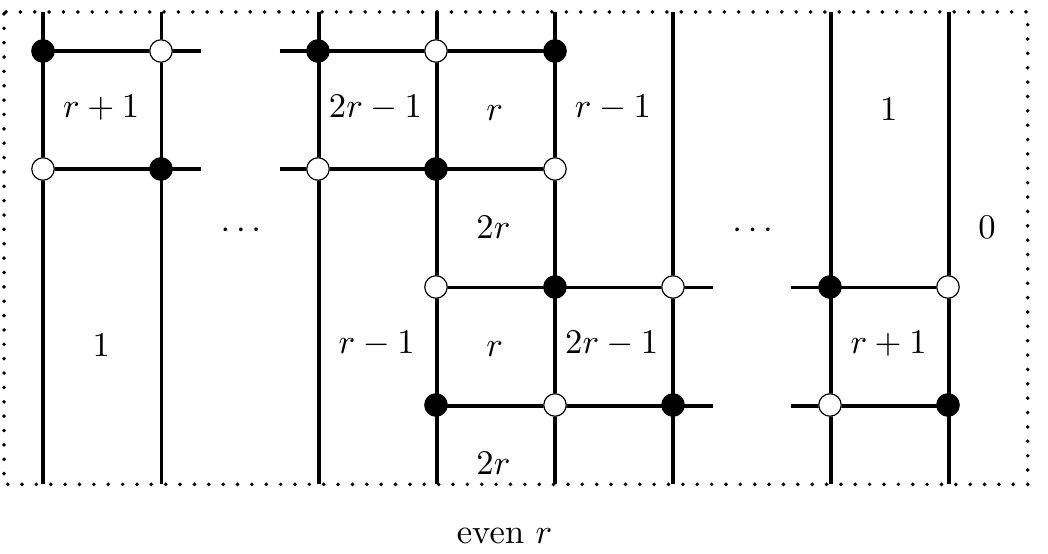}
\end{center}
\begin{center}
\includegraphics[width=0.6\textwidth]{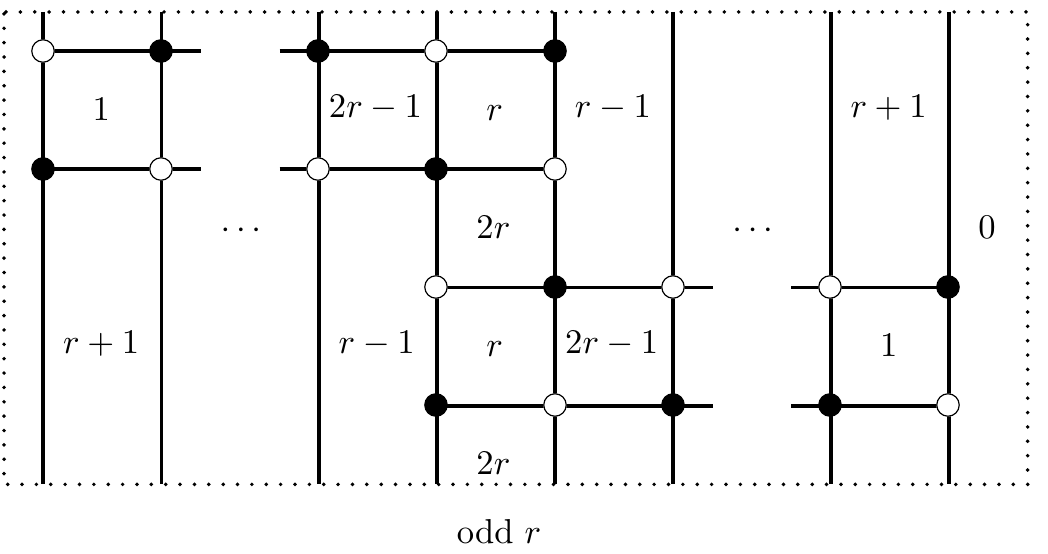}
\end{center}

We then assign Q-system variables as weights of $\G$ according to face labels, see Theorem~\ref{thm:qsys_cluster}.
We abuse the notation and use $\mu$ for a sequence of mutations where the mutation at $i$ actually means the mutations at both faces labeled by $i$.
Then $\sigma\mu$ sends $\Q$ to itself and the Q-variables are shifted by $k\rightarrow k+1$ as well.
To be precise, $Q_{i,k} \mapsto Q_{i,k+1}$ for $i\in[1,r-1]$ while $Q_{r,2k}\mapsto Q_{r,2k+2}$ and $Q_{r,2k+1}\mapsto Q_{r,2k+3}.$

Let $M_0$ be the perfect matching of $\G$ containing all vertical edges whose top vertex is black.
It can be depicted as the following when $r$ is odd and similarly when $r$ is even.
\begin{center}
\includegraphics[width=0.6\textwidth]{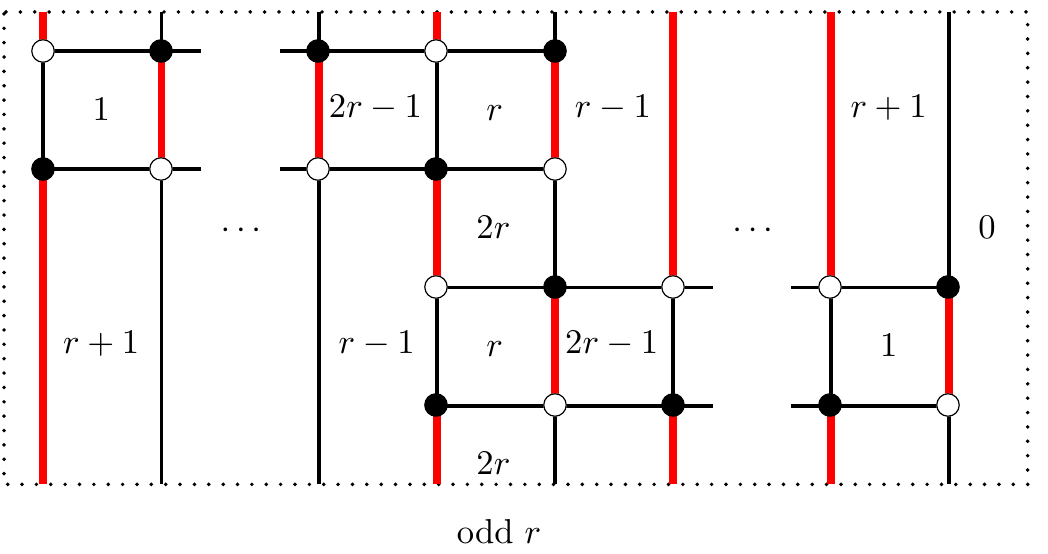}
\end{center}

Similarly to the proof of Theorem~\ref{thm:Ar_conserve}, we can check that all the conditions in Theorem~\ref{thm:mutation_mutation} and Theorem~\ref{thm:hamiltonian_conserved} hold.
Hence we have
\[ H_{(i,j),\G,M_0}(A_{1,k},\dots,A_{2r,k}) = H_{(i,j),\G,M_0}(A_{1,k+1},\dots,A_{2r,k+1}) \]
where
\[ A_{i,k} = \begin{cases} Q_{i,k}, & i\in[1,r-1],\\ Q_{r,2k}, & i = r, \\ Q_{i-r,k+1}, & i\in[r+1,2r-1], \\ Q_{r,2k+1}, & i=2r.\end{cases} \]
Hence the Hamiltonians $H_{(i,j),\G,M_0}$ are conserved quantities of the $B_r$ Q-system.
This proved the following theorem.

\begin{thm}
Let $(\G,(A_{i}))$ be a weighted bipartite torus graph defined above.
Then the Hamiltonians $H_{(i,j),\G,M_0}(A_1,\dots,A_{2r})$ are conserved quantities of the $B_r$ Q-system dynamic $Q_{\a,k} \mapsto Q_{\a, k + t_\a }$ where $t_\a =1 $ for $\a \in [1,r-1]$ and $t_r = 2$.
\end{thm}


\subsection{Partition function of hard particles}

In this section, we write the Hamiltonians as partition functions of hard particles on a weighted graph, analog to what has been done for Q-systems of type A.

From Proposition~\ref{prop:matching_loop}, $[M]-[M_0]$ is a product of non-intersecting simple loops of $\G$.
The following loops are all connected simple loops that can be appeared (given our choice of $M_0$).
This will be proved in Theorem~\ref{thm:Br_matching_gamma}.

We first define $2r$ straight loops on $\G$ and denote them by $\Gamma_{2a-1}$ for $a\in [1,2r]$.
\begin{center}
\includegraphics[scale=0.8]{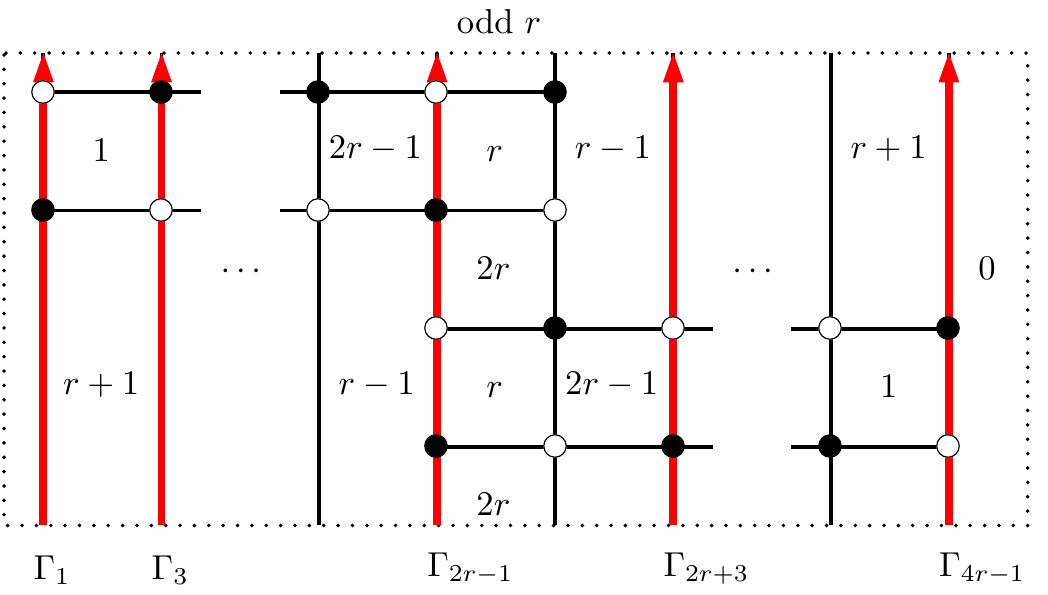}
\end{center}

Next, we define zig-zag loops $\Gamma_{2a}$ for $a\in[1,r-1]\cup [r+1,2r-1]$ as follows.
\begin{center}
\includegraphics[scale=0.8]{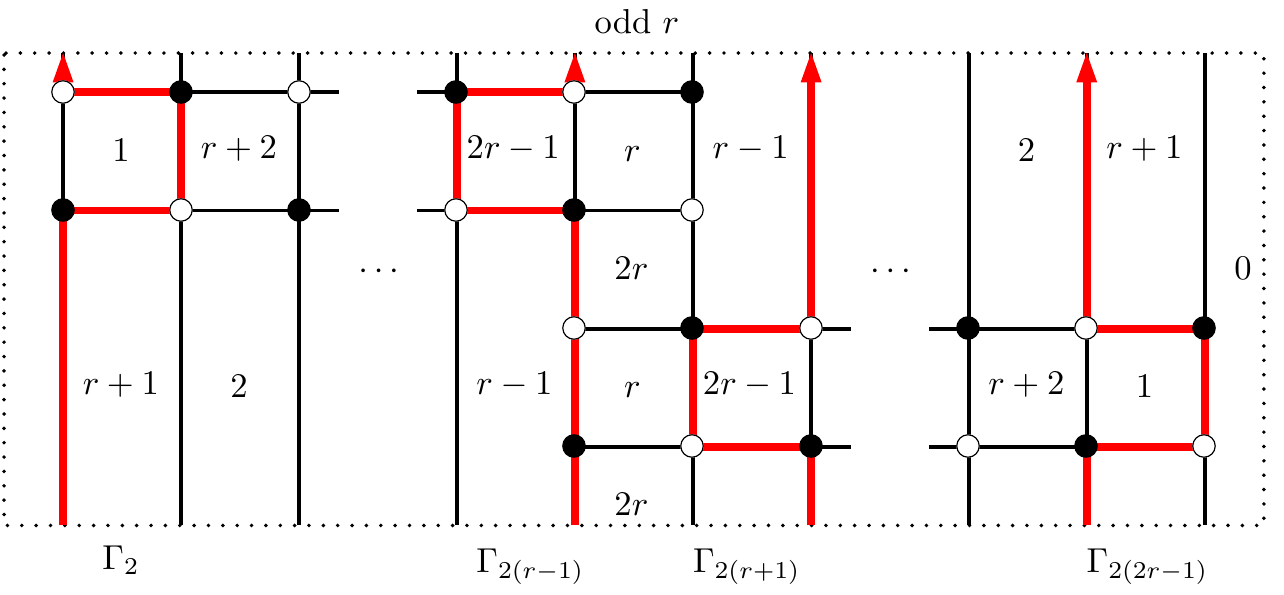}
\end{center}
We notice that when $a\in[1,r-1]$, $\Gamma_{2a}$ always goes counterclockwise around face $a$ and clockwise around face $r+a$.
When $a \in [r+1,2r-1]$, $\Gamma_{2a}$ goes counterclockwise around face $2r-a$ and clockwise around face $3r-a$.

Lastly, we define $\Gamma_{2r,j}$ for $j\in[1,6]$.
They are depicted as follows.
\begin{center}
\includegraphics[scale=.8]{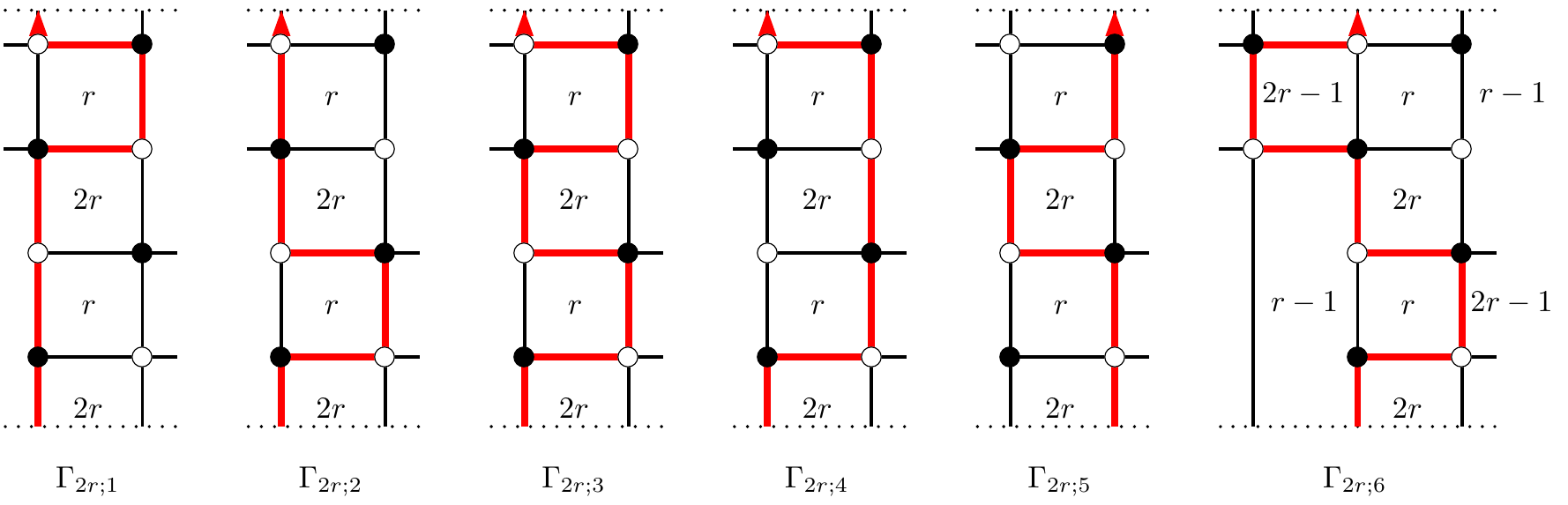}
\end{center} 

We then have the following theorem analog to Theorem~\ref{thm:Ar_matching_gamma}.

\begin{thm} \label{thm:Br_matching_gamma}
Let $M$ be a perfect matching of $\G$. 
Then $[M]-[M_0]$ is a nonintersecting collection of $\Gamma$.
Furthermore, every nonintersecting collection of $\Gamma$ is $[M]-[M_0]$ for a unique perfect matching $M$ of $\G$.
\end{thm}
\begin{proof}
Using exactly the same proof as in Theorem~\ref{thm:Ar_matching_gamma}, we can see that all the loops $\Gamma$ listed above are all possible simple loops appeared in $[M]-[M_0]$.
\end{proof}

Let $G_r$ be the following graph with $4r+4$ vertices indexed by the set $[1,2r-1]\cup[2r+1,4r-1] \cup \{(2r,i)\mid i\in[1,6]\} $ defined as the following
\begin{center}
\includegraphics[scale = .8]{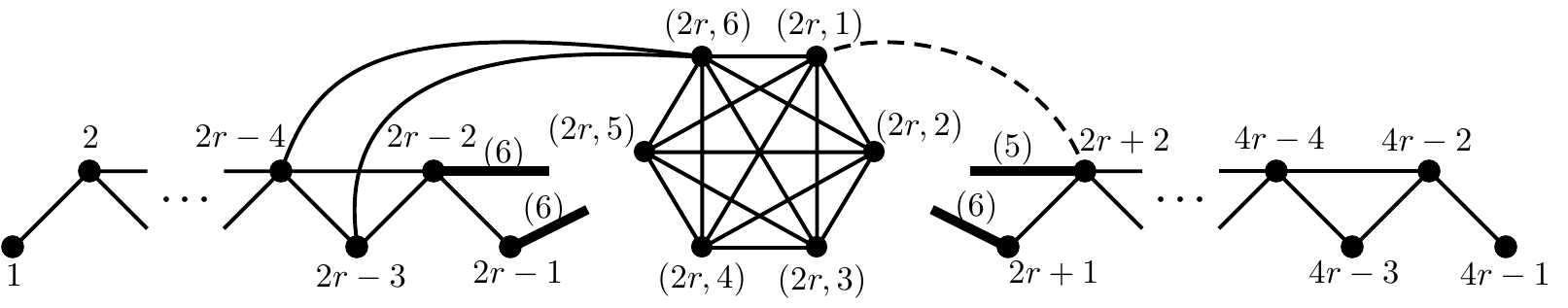}
\end{center}
There is a complete graph $K_6$ as a subgraph of $G_r$ in the middle.
The label (6) on the three thick lines indicates that the vertex connects to all six vertices in $K_6$.
There is a thick line with label (5);
vertex $2r+2$ connects to all vertices in $K_6$ except to the vertex $(2r,1)$.
(This is indicated by a dotted line between $(2r,1)$ and $(2r+2)$.)
The vertex $(2r,6)$ has two extra edges connecting to vertex $2r-4$ and $2r-3$.

The loop $\Gamma_i$ intersects $\Gamma_j$ if and only if the vertices $i$ and $j$ are connected in $G_r$.
For any subset $I\subseteq [1,2r-1]\cup[2r+1,4r-1] \cup \{(2r,i)\mid i\in[1,6]\}$, the loops in $\{ \Gamma_i \mid i\in I\}$ are pairwise disjoint if and only if $I$ is a subset of pairwise nonadjacent vertices of $G_r$.
Also, nonintersecting collections of $\Gamma$ of size $n$ are in bijection with $n$-subsets of pairwise nonadjacent vertices of $G_r$.

By Theorem~\ref{thm:Br_matching_gamma}, the possible homology classes of $[M]-[M_0]$ are $(0,k)$ for $k\in[0,2r]$.
Let $\gamma_i$ be the weight of $\Gamma_i$.
Theorem~\ref{thm:Br_matching_gamma} implies the following theorem.

\begin{thm} \label{thm:hamiltonian_gamma_Br}
Let $k\in[0,2r]$. Then
\[ H_{(0,k),\G,M_0}(A_{1},\dots,A_{2r}) = \sum_{|I|=k}\prod_{i\in I} \gamma_i \]
where the sum runs over all $k$-subsets $I$ of pairwise nonadjacent vertices of $G_r$.
\end{thm}

\begin{ex}
When $r=2$, the graph $G_2$ for $B_2$ Q-system has 12 vertices and can be depicted as the following.
\begin{center}
\includegraphics[scale=.8]{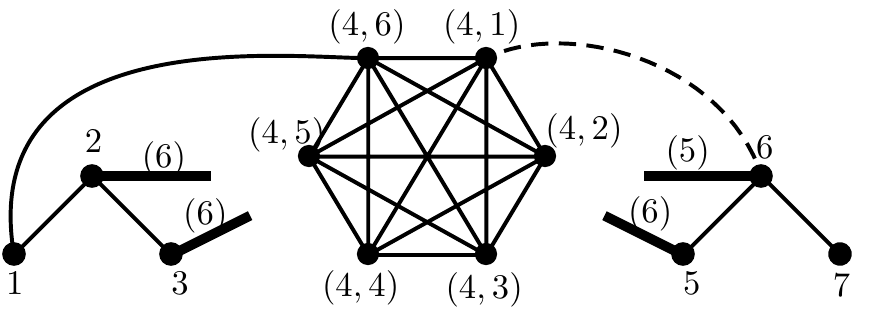}
\end{center}
We have the following weights.
\begin{align*}
\gamma_1 &= \frac{Q_{1,k+1}}{Q_{1,k}}, & 
\gamma_2 &= \frac{Q_{2,2k+1}^2}{Q_{1,k} Q_{1,k+1} Q_{2,2k}}, &
\gamma_3 &= \frac{Q_{1,k} Q_{2,2k+1}^2}{Q_{1,k+1} Q_{2,2k}^2}, \\
\gamma_{4;1} &= \frac{Q_{1,k}^2}{Q_{2,2k}^2}, &
\gamma_{4;2} &= \frac{Q_{1,k}^2}{Q_{2,2k}^2}, &
\gamma_{4;3} &= \frac{Q_{1,k}^3 Q_{1,k+1}}{Q_{2,2k}^2 Q_{2,2k+1}^2}, \\
\gamma_{4;4} &= \frac{Q_{1,k} Q_{1,k+1}}{Q_{2,2k+1}^2}, &
\gamma_{4;5} &= \frac{Q_{1,k} Q_{1,k+1}}{Q_{2,2k+1}^2},  &
\gamma_{4;6} &= \frac{1}{Q_{2,2k}}, \\
\gamma_5 &= \frac{Q_{1,k+1} Q_{2,2k}^2}{Q_{1,k} Q_{2,2k+1}^2}, &
\gamma_6 &= \frac{Q_{2,2k}}{Q_{1,k} Q_{1,k+1}}, &
\gamma_7 &= \frac{Q_{1,k}}{Q_{1,k+1}}.
\end{align*}
For $H_k := H_{(0,k),G,M_0}(A_{1,k},\dots,A_{2r,k})$, we have
\begin{align*}
H_0 &= 1 \\
H_1 &= \frac{Q_{1,k+1}}{Q_{1,k}}+\frac{Q_{2,2k+1}^2}{Q_{1,k} Q_{1,k+1} Q_{2,2k}}+\frac{Q_{1,k} Q_{2,2k+1}^2}{Q_{1,k+1} Q_{2,2k}^2}+\frac{Q_{1,k}^2}{Q_{2,2k}^2}+\frac{Q_{1,k}^2}{Q_{2,2k}^2}+\frac{Q_{1,k}^3 Q_{1,k+1}}{Q_{2,2k}^2 Q_{2,2k+1}^2}+\\
&+\frac{Q_{1,k} Q_{1,k+1}}{Q_{2,2k+1}^2}+\frac{Q_{1,k} Q_{1,k+1}}{Q_{2,2k+1}^2}+\frac{1}{Q_{2,2k}}+\frac{Q_{1,k+1} Q_{2,2k}^2}{Q_{1,k} Q_{2,2k+1}^2}+\frac{Q_{2,2k}}{Q_{1,k} Q_{1,k+1}}+\frac{Q_{1,k}}{Q_{1,k+1}}, \\
H_2 &= \frac{Q_{1,k}^4}{Q_{2,2k}^2 Q_{2,2k+1}^2}+\frac{2 Q_{1,k}^3}{Q_{1,k+1} Q_{2,2k}^2}+\frac{Q_{2,2k+1}^2 Q_{1,k}^2}{Q_{1,k+1}^2 Q_{2,2k}^2}+\frac{2 Q_{1,k}^2}{Q_{2,2k+1}^2}+\frac{Q_{1,k+1}^2 Q_{1,k}^2}{Q_{2,2k}^2 Q_{2,2k+1}^2}+\\
&+\frac{2 Q_{1,k}}{Q_{1,k+1} Q_{2,2k}}+\frac{2 Q_{1,k+1} Q_{1,k}}{Q_{2,2k}^2}+\frac{Q_{2,2k+1}^2}{Q_{1,k}^2 Q_{1,k+1}^2}+\frac{2 Q_{2,2k+1}^2}{Q_{1,k+1}^2 Q_{2,2k}}+\frac{Q_{2,2k+1}^2}{Q_{2,2k}^2}+\\
&+\frac{2 Q_{2,2k}}{Q_{1,k}^2}+\frac{2 Q_{1,k+1}^2}{Q_{2,2k+1}^2}+\frac{Q_{2,2k}^2}{Q_{2,2k+1}^2}+\frac{Q_{1,k+1}^2 Q_{2,2k}^2}{Q_{1,k}^2 Q_{2,2k+1}^2}+2, \\
H_3 &= H_1, \\
H_4 &= 1.
\end{align*}
We notice that $H_0 = H_4$ and $H_1 = H_3$.
In fact, we will show that $H_i = H_{2r-i}$ ($i\in[0,2r]$) for the Hamiltonians of the $B_r$ Q-system.
\end{ex}

\begin{prop}
Let $H_k := H_{(0,k),G,M_0}(A_1,\dots,A_{2r})$ be the Hamiltonians for the $B_r$ Q-system.
Then for $ k\in[0,2r],$
\[ H_k = H_{2r-k}. \]
\end{prop}
\begin{proof}
Fix $r \geq 2$.
Let $\pi : [1,2r-1]\cup \{ (2r;j)\}_{j=1}^6 \cup [2r+1,4r-1] \rightarrow [1,2r]$ be a projection defined by
\[ \pi:\begin{cases} i \mapsto i, & i\in[1,2r-1],\\ (2r;j) \mapsto 2r, & j\in[1,6],\\ i\mapsto 4r-i, &i\in[2r+1,4r-1]. \end{cases}
\]
Let $F_r$ be the following graph with $2r$ vertices.
\begin{center}
\includegraphics[scale=0.8]{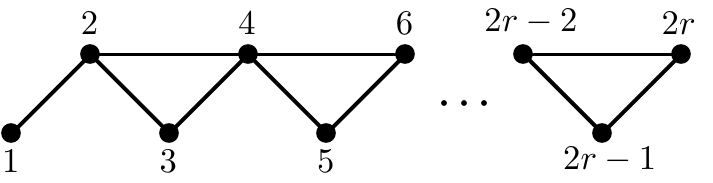}
\end{center}
Conceptually $F_r$ is obtained from $G_r$ by collapsing all six vertices $(2r;1),\dots,(2r;6)$ into one vertex $2r$ and folding the resulting graph by identifying vertices $i \leftrightarrow 4r-i$.
The projection $\pi$ sends a vertex of $G_r$ to a vertex of $F_r$ by the collapsing/folding procedure.

Let $\mathcal{C}$ be the set of all hard particle configurations on $G_r$, and let $\mathcal{C}_k$ be the set of all $k$ hard particle configurations on $G_r$.
We have $\mathcal{C}_k \neq \emptyset$ for $k\in[0,2r]$. (The set $\mathcal{C}_0$ contains exactly one configuration, the empty configuration.)
So
\[ \mathcal{C} = \bigsqcup_{k=0}^{2r} \mathcal{C}_k.\]
We identify a hard particle configuration with a subset of $[1,2r-1]\cup \{ (2r;j)\}_{j=1}^6 \cup [2r+1,4r-1]$.
Let $A\in\mathcal{C}$.
We then associate each element $i\in A$ with a black or white dot on the vertex $\pi(i)$ $F_r$ as follows.
\begin{itemize}
\item An element $i\in [1,2r-1]$ is associated with a black dot on vertex $\pi(i)$.
\item An element $i\in [2r+1,4r-1]$ is associated with a white dot on vertex $\pi(i)$.
\item An element $(2r;j)$, for $j\in[1,5]$, is associated with a black dot together with a label $j$ on vertex $2r$.
\item The element $(2r,6)$ is associated with a white dot together with a label $6$ on vertex $2r$ and a black dot on vertex $2r-2$.
\end{itemize}
See an example in Figure~\ref{fig:involution_folding_ex}.
\begin{figure}
\includegraphics[scale=0.6]{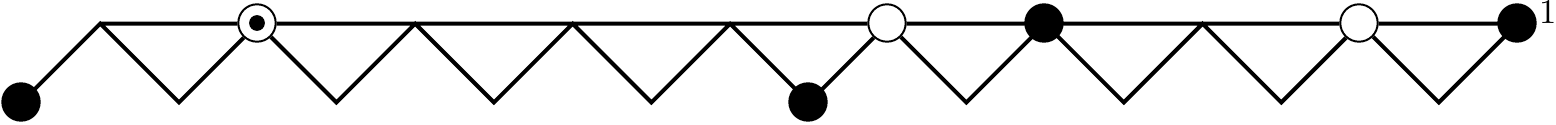}
\caption{An example of dots on $F_{10}$ associated with a configuration $\{1,4,11,14,20,27,33,41\}$ in $\mathcal{C}$.
The dot $\odot$ indicates two dots of different colors on the same vertex.}
\label{fig:involution_folding_ex}
\end{figure}
Notice that we can recover the hard particle configuration on $G_r$ from a configuration of dots on $F_r$.
It is obvious that not every dot configuration on $F_r$ is associated with a configuration in $\mathcal{C}$.

We define an involution $\iota:\mathcal{C}\rightarrow\mathcal{C}$ by the following steps.
\begin{enumerate}
\item Let $A\in\mathcal{C}$. Consider the dot configuration on $F_r$ associated with $A$.
We have a collection of connected chains of dots.
Since $A$ is a hard particle configuration, each chain is a chain of dots of alternate colors or A pair of white and black dots on the same vertex.

For each connected chain $i_1<\dots< i_k$ or a pair of black and white dots $i_1=i_2$ ($k=2$), we let $i$ be the largest odd integer such that $i\leq i_1$, $j$ be the smallest even integer such that $i_k < j$.
When $j>2r$, we set $j=2r$.
We then delete edges $(i-1,i)$, $(i-1,i+1)$, $(j,j+1)$ and $(j,j+2)$ of $F_r$ when possible.
As a result, $F_r$ is decomposed into disconnected blocks (isomorphic to $F_{n}$, $n\leq r$).
Lastly any block of $2n$ vertices containing no dots is decomposed into $n$ blocks of $2$ vertices.
See Figure~\ref{fig:Fr_break} for an example.
\item For each block we define an involution on dot configurations as in Figure~\ref{fig:dot_inv1} and \ref{fig:dot_inv2}.
The blocks in Figure~\ref{fig:dot_inv_fix} are fixed by the involution.
\item The hard particle configuration $\iota(A)$ is the configuration associated with the resulting dot configuration.
See Figure~\ref{fig:config_F_2} for an example.
\end{enumerate}

\begin{figure}
\includegraphics[scale=0.6]{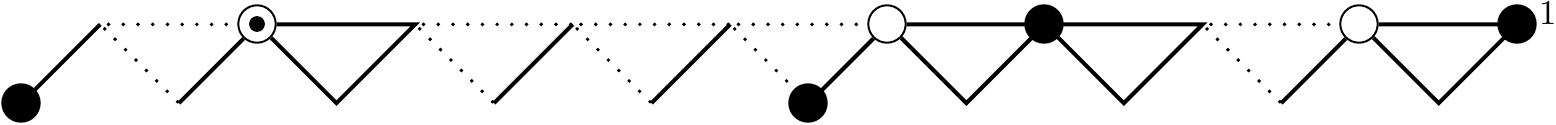}
\caption{A running example from Figure.~\ref{fig:involution_folding_ex} showing $F_{10}$ decomposed into blocks.}
\label{fig:Fr_break}
\end{figure}
\begin{figure}
\includegraphics[scale=0.6]{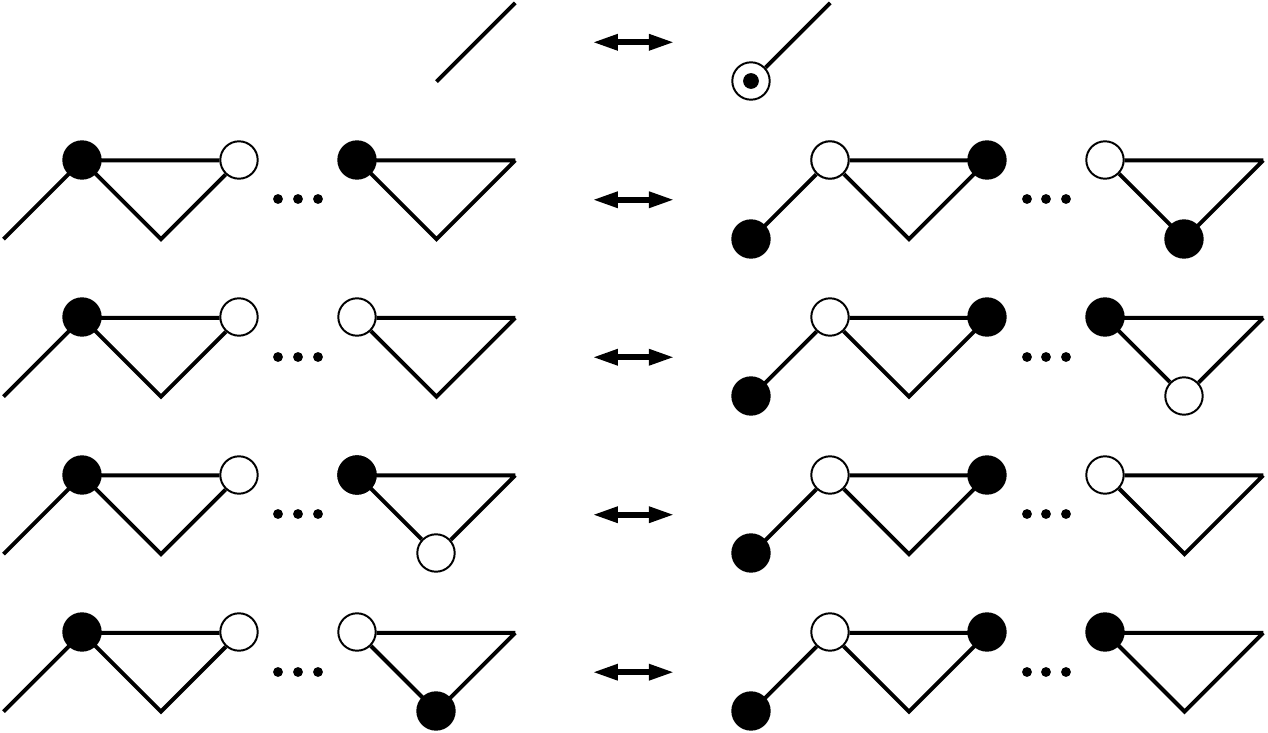}
\caption{An involution on dot configurations when the vertex $2r$ is unoccupied.
We have another version of each picture when all of the colors are switched.}
\label{fig:dot_inv1}
\end{figure}
\begin{figure}
\includegraphics[scale=0.6]{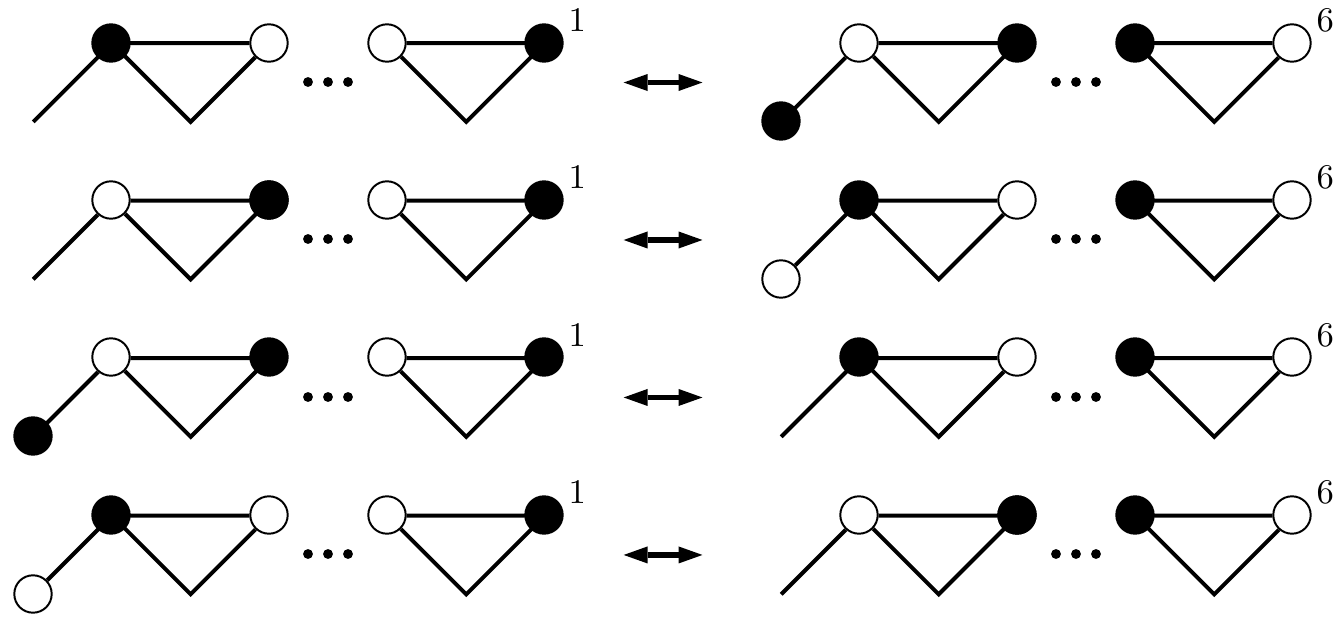}
\caption{An involution on dot configurations when the vertex $2r$ is occupied.}
\label{fig:dot_inv2}
\end{figure}
\begin{figure}
\includegraphics[scale=0.6]{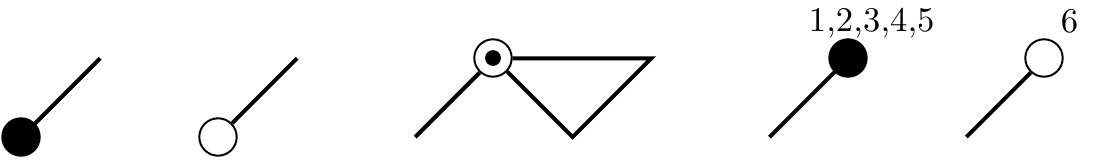}
\caption{The blocks which are fixed by the involution.}
\label{fig:dot_inv_fix}
\end{figure}
\begin{figure}
\includegraphics[scale=0.6]{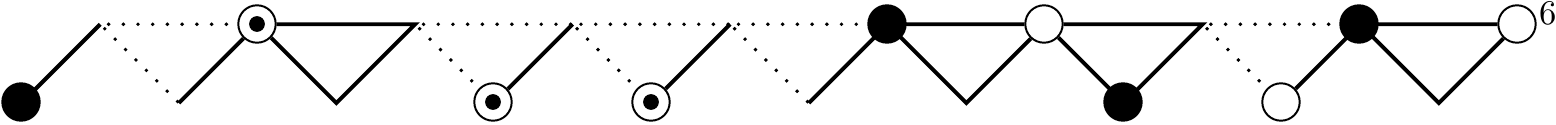}
\caption{The resulting dot configuration after an involution of a configuration $A$ in Figure~\ref{fig:involution_folding_ex}.
So we have $\iota(A)=\{1,4,7,9,12,15,25,28,31,36,38,41\}$.}
\label{fig:config_F_2}
\end{figure}

For a block of $2n$ vertices the involution sends a dot configuration of size $k$ to a configuration of size $2n-k$, where we do not count the black dot at $2r-2$ when the white dot with label $6$ on $2r$ is present.
The sum of all the number of vertices of all blocks is $2r$. 
So the involution is a bijection between $\mathcal{C}_k$ and $\mathcal{C}_{2r-k}$.

It is left to show that if $A\in \mathcal{C}_k$ and $B=\iota(A) \in \mathcal{C}_{2r-k}$ then 
\begin{align}\label{eq:weight_involution}
\prod_{i\in A} w(\gamma_i) = \prod_{j\in B} w(\gamma_j).
\end{align}
By a direct computation, we have
\begin{align*}
\gamma_{4r-(2a-1)}\gamma_{2a} &= \gamma_{4r-2a}\gamma_{2a+1}, &
\gamma_{2a-1}\gamma_{4r-2a} &= \gamma_{2a}\gamma_{4r-(2a+1)}, \\
\gamma_{2b-1}\gamma_{4r-(2b-1)} &= 1 , &
\gamma_{(2r;1)}\gamma_{2r+2} &= \gamma_{(2r+6)}\gamma_{2r+3},
\end{align*}
for $a\in[1,r-1]$ and $b\in[1,r]$.
These equations say that the following moves preserve the weight of the associated hard particle configuration.
\begin{align} \label{eq:ele_move}
\raisebox{-0.4\height}{\includegraphics[scale=0.6]{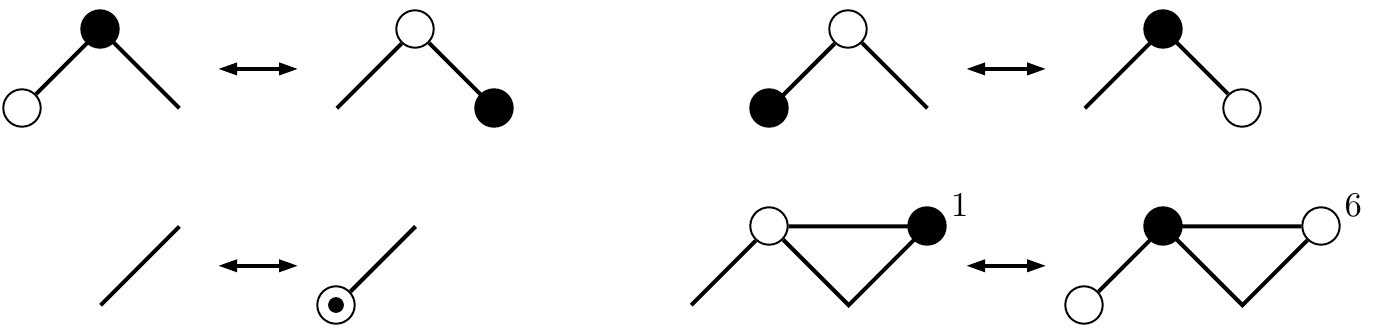}}
\end{align}
Since the involution on each block can be written as a sequence of the moves \eqref{eq:ele_move} (see Figure~\ref{fig:ex_ele_move} for an example), the involution preserves the weight for each block.
Hence the equation \eqref{eq:weight_involution} holds.
\begin{figure}
\includegraphics[scale=0.6]{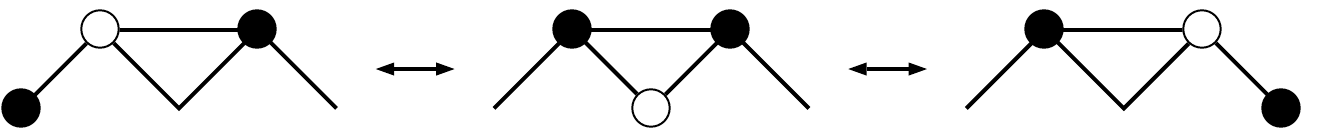}
\caption{The involution on a block can be written as a sequence of moves in \eqref{eq:ele_move}.}
\label{fig:ex_ele_move}
\end{figure}
\end{proof}

\subsection{Poisson bracket} \label{sec:bracket_b}

Similar to type $A$, we let let $C$ be the Cartan matrix of type $B_r$.
The signed adjacency matrix of the quiver of $B_r$ Q-system is
\begin{align*}
B = \begin{bsmallmatrix}C-C^T & C^T \\ -C & 0\end{bsmallmatrix}.
\end{align*}
Let $\X$ be a phase space with coordinates $(A_1,\dots,A_{2r})$.
Define a Poisson bracket on the algebra $\mathcal{O}(\X)$ of functions on $\X$ by
\[ \{ A_i , A_j \} = \Omega_{ij} A_i A_j  \quad (i,j\in [1,2r])\]
where the coefficient matrix $\Omega$ is defined by
\begin{align*}
\Omega = (B^T)^{-1} = -B^{-1}.
\end{align*}

Comparing to Proposition~\ref{prop:A_poisson_pairing}, the Poisson bracket for $B_r$ Q-system cannot be written as the intersection pairing.
This is due to the existence of faces with the same weight, which makes the Poisson bracket not local.
Nevertheless, experimental data still show that the Hamiltonians Poisson-commute.

\begin{conj} \label{conj:hamiltonian_commute_B}
The Hamiltonians of the $B_r$ Q-system Poisson-commute.
\end{conj}


\section{Dimer integrable systems}\label{sec:gk}

In this section, we compare our constructions and results to \cite{GK13}.
First we summarize the result from \cite{GK13}.



\subsection{Minimal bipartite torus graphs from convex polygons}\label{sec:bipartitegraph}

Let $N$ be a convex polygon in $\R^2$ with corners in $\Z^2$, called \defemph{integral polygon}, considered up to translation by vectors in $\Z^2$.
We pick all the integral vertices on the boundary of $N$ (i.e. every vertex in $\partial N \cap \Z^2$) counterclockwise, and get a sequence of vertices $v_1,v_2,\dots,v_n$ where $n$ is the number of integral vertices on $\partial N$ and the indices are read modulo $n$.
Let vectors $e_i$ be vectors pointing from $v_i$ to $v_{i+1}$.
We get from the construction that each $e_i = (a_i,b_i)$ is a primitive vector, i.e. $a_i,b_i \in\Z$ and $\gcd (a_i,b_i)=1$.
We then get a collection $\{e_i\}$ of integral primitive vectors in $\Z^2$.

Consider the torus $\T=\R^2/\Z^2$.
Each $e_i$ determines a homology class $(a_i,b_i)\in H_1(\T,\Z)=\Z\times \Z$, and there is a unique up to translation geodesic representing this class.
In other words, it is an oriented straight line on $\T$ with slope $b_i/a_i$, i.e. a projection of $e_i$ on the torus $\T $.
Note that the geodesics are indeed oriented loops on $\T$ since their slopes are rational.

We then take a family of distinct oriented loops $\{\a_i\}$ on $\T$ such that the isotopy class of $\a_i$ matches the isotopy class of the geodesic representing $e_i$.
By Theorem~\ref{thm:graph_move_GK} we can choose $\{\a_i\}$ such that the loops are in generic position (no intersection of more than two loops) and satisfy the following conditions \cite[Definition 2.2]{GK13}:
\begin{enumerate}
  \item (admissibility) Going along any loops $\a_i$, the directions of the other loops intersecting it alternate (left-to-right or right-to-left).
  \item \label{cond:min}(minimality) The total number of intersections is minimal.
\end{enumerate}
The collection $\{\a_i\}$ provides a decomposition of $\T$ into a union of polygons whose oriented sides are parts of $\a_i$ and vertices are intersection points of the loops $\a_i$.
Then the first condition is equivalent to
\begin{enumerate}
  \item[(1\ensuremath{^\prime})] (admissibility) The sides of any polygon $P_i$ are either oriented clockwise, counterclockwise or alternate.
\end{enumerate}
The family of oriented loops $\{\a_i\}$ gives rise to an oriented graph on the torus $\T$.
We call an oriented graph on $\T$ satisfying the above conditions \defemph{minimal admissible graph on a torus}.

\begin{ex}
Starting from an integral polygon on the left, we obtained four primitive vectors depicted in the middle picture.
\begin{center}
\includegraphics[scale=0.7]{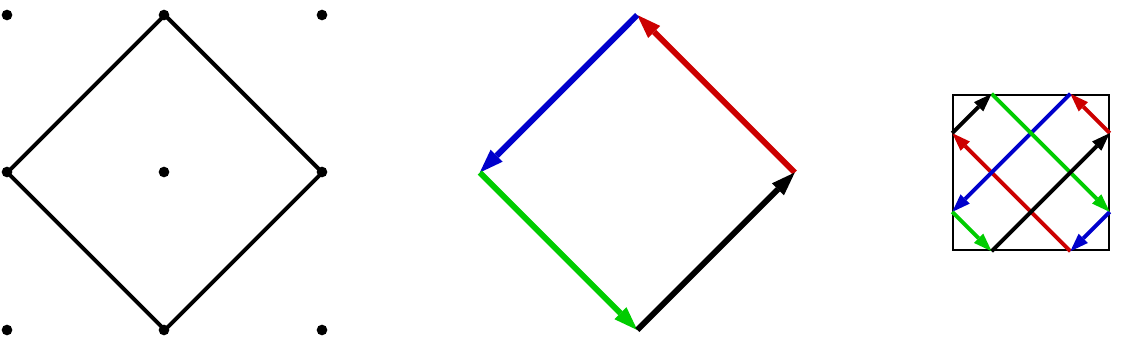}
\end{center}
The vectors are then associated with oriented loops on a torus, which gives a minimal admissible graph shown in the picture on the right.
\end{ex}

Given an admissible minimal torus graph, we construct a bipartite graph $\G$ on $\T$ by constructing vertices from the polygons $P_i$ where the coloring is determined as follows:
\begin{itemize}
  \item Polygons $P_i$ with counterclockwise orientation are associated with black vertices.
  \item Polygons $P_i$ with clockwise orientation are associated with white vertices.
\end{itemize}
From the construction, every intersection is a vertex shared by exactly two well-oriented polygons having opposite orientations.
We then associate each shared vertex with an edge connecting the two vertices of $\G$ associated to the two polygons.

We see that $\G$ is indeed a bipartite graph, and we will call it a \defemph{minimal bipartite torus graph}.
For the rest of the section, unless stated otherwise we assume that $G$ is a minimal bipartite torus graph obtained from an integral polygon $N$.

\begin{ex}
In our running example, there are eight polygons and eight intersection points.
\begin{center}
\includegraphics[scale=0.8]{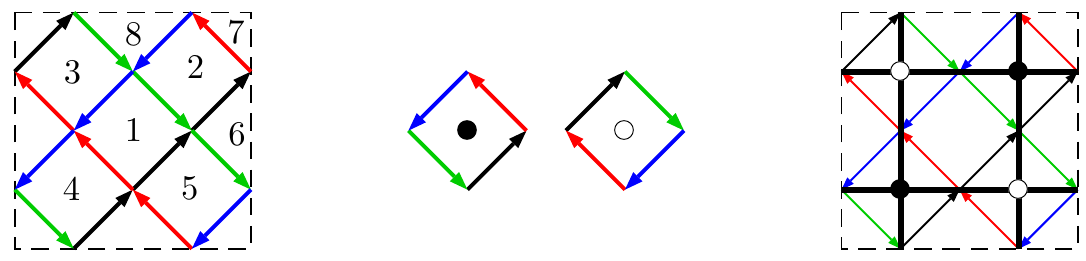}
\end{center}
In the picture, the polygons labeled by $2,3$ (resp. $4,5$) are oriented counterclockwise (resp. clockwise), so they are associated with black (resp. white) vertices of $\G$.
The polygons labeled by $1,6,7,8$ have alternate orientation.
The eight intersection points are associated with eight edges of $\G$.
\end{ex}

From the bipartite graph $G$, we can uniquely (up to translation) recover the starting integral polygon by reversing the process.
The convexity condition on the integral polygon will guarantee the uniqueness of the polygon.
We also note that a 180-degree rotation of $N$ corresponds to reversion of the orientation of all loops in the admissible minimal graph.
This will switch the color of the vertices of $\G$.

%
%
%
%

For an arbitrary integral polygon $N$, a minimal admissible graph on a torus associated with $N$ and a minimal bipartite torus graph associated with $N$ always exist.
Furthermore, we can obtain one minimal bipartite torus graph from another by use of the two elementary moves (definitions~\ref{def:urban} and \ref{def:shrink}), as stated in the following theorem.

\begin{thm}[{\cite[Theorem 2.5]{GK13}}]
\label{thm:graph_move_GK}
For any integral polygon $N$ there exists a minimal admissible graph on a torus associated with $N$.
It produces a minimal bipartite torus graph $\G$ associated with $N$.
Furthermore, any two minimal bipartite graphs on a torus associated with $N$ are related by a sequence of urban renewals and shrinking of 2-valent vertices.
\end{thm}

\begin{remark}
Our bipartite torus graphs for the Q-systems of type A and B in sections~\ref{sec:A} and \ref{sec:B} are not obtained from non-degenerate convex polygons.

For $A_r$ Q-system, we reverse the process and construct oriented loops on the torus as in the following picture.
\begin{center}
\includegraphics[scale=0.7]{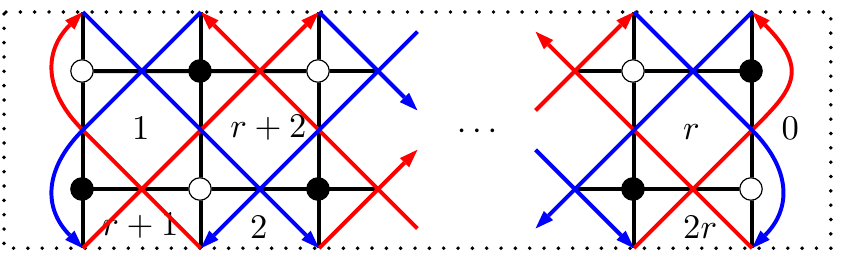}
\end{center}
We get only two oriented loops depicted in blue and red whose homology classes are $(0,-r-1)$ and $(0,r+1)$, respectively.
Notice that they are not primitive, and they form a vertical degenerate bigon with sides of length $r+1$.

For $B_r$ Q-system, we have the following picture.
\begin{center}
\includegraphics[scale=0.7]{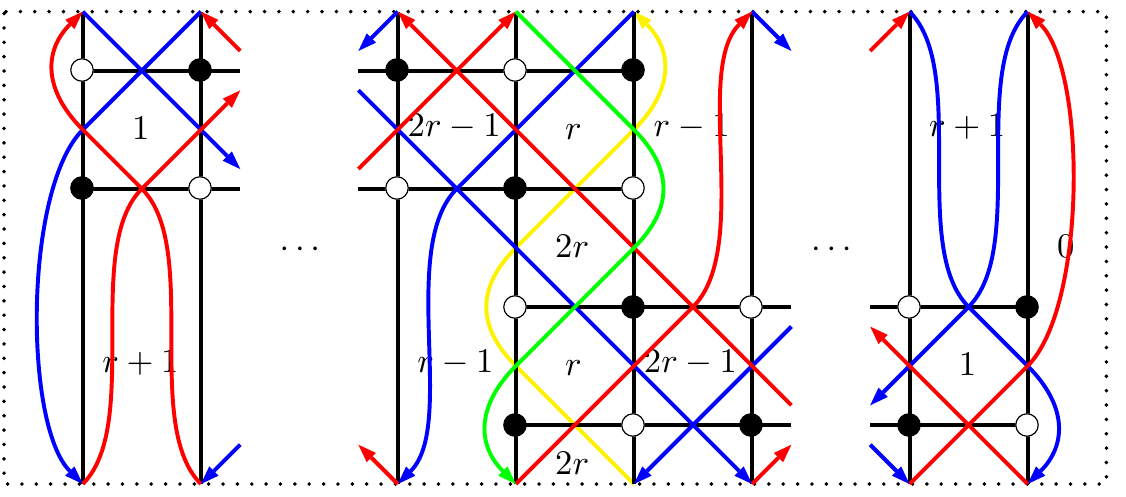}
\end{center}
There are four oriented loops depicted in blue, green, yellow and red whose homology classes are $(0,-2r+1),(0,-1),(0,1)$ and $(0,2r-1)$, respectively.
The blue and red loops are not primitive when $r > 1$.
The loops form a vertical degenerate quadrilateral.
\end{remark}


\subsection{Phase space and Poisson structure}\label{sec:manifold}

For a minimal bipartite torus graph $\G$ with $n$ faces, let $\L_\G$ be the moduli space of line bundles with connections on $\G$.
We have $\L_\G \cong \Hom(H_1(\G,\Z),\C^*)$, so combinatorially $\L_\G$ is the set of all weight assignments to all the loops on $\G$ compatible with loop multiplication. (The weight of a product of loops coincides with the product of their weights.)

For $j\in[1,n]$, let $y_j$ be the weight assigned to the counterclockwise loop $Y_j$ around the face $j$ of $\G$.
Since $\G$ is a graph embedded on a torus, there is a projection $H_1(\G,\Z)\rightarrow H_1(\T,\Z)$.
We then pick two loops $Z_1,Z_2$ having homology classes $(1,0),(0,1)\in H_1(\T,\Z)$ under the projection, and assign weight $z_1,z_2$ to them, respectively.
Any loop on $\G$ can then be generated by $y_j$'s together with $z_1$ and $z_2$, where a product of the variables corresponds to a product of loops.
Since the product of all the face loops is trivial, we must have $\prod_{j=1}^n w_j = 1$.
This is the only condition among the generators.
So $\dim \L_\G = n+1$ and the algebra $\mathcal{O}(\L_\G)$ of functions on $\L_\G$ has $(y_1,\dots,y_{n-1},z_1,z_2)$ as coordinates.

Note that this weight is different from our weight in Definition~\ref{def:weight_path}.
The connection between the two is discussed in Remark~\ref{rem:weight}.

For any loops $\Gamma_1,\Gamma_2$ on $\G$, the Poisson bracket of their weights is defined in terms of the intersection pairing of the twisted ribbon graph associated with $\G$ as in Definition~\ref{def:int_pairing}.

Now we define a Y-seed $(B,(y_1,\dots,y_n,z_1,z_2))$ of rank $n+2$ associated with $\L_\G$, where the exchange matrix $B = (B_{ij})$ and 
\[  B_{ij} = \epsilon(Y_i,Y_j) \quad\text{for }i,j\in[1,n+2] \]
with $Y_{n+1} := Z_1$ and $Y_{n+2} := Z_2$. 

Let $\G$ and $\G'$ be two minimal bipartite torus graphs associated with the same integral polygon $N$.
By Theorem~\ref{thm:graph_move_GK}, they are related by a sequence of elementary moves.
These moves induce an isomorphism $i_{\G,\G'}:\L_G \rightarrow \L_{G'}$ according to Y-seed mutations.
Let $\X_N$ be the \defemph{phase space} defined by gluing the spaces $\L_G$ by the isomorphisms.
The phase space depends only on $N$ and each isomorphism $i_{\G,\G'}$ can be viewed as a change of coordinate.

\begin{remark}
\label{rem:weight}
Let $G$ be a minimal bipartite torus graph with $n$ faces.
Each choice of face weight $(A_i)\in (\C^*)^n$ induces a weight assignment on oriented edges of $G$ by Definition~\ref{def:weight_path}.
This then induces a weight assignment on $Y_1,\dots,Y_n,Z_1,Z_2,$ hence a loop weight in $\L_G.$
Since $\L_\G$ has dimension $n+1$, not all loop weights of \cite{GK13} can be obtained from our weight in Definition~\ref{def:weight_path}.

In addition, the weight $y_j$ around face $j$ of $G$ is the $j^\text{th}$ $\tau-$coordinates (Definition~\ref{def:can_tran}) of a cluster seed $(\mathbf{A},B)$, i.e.
\[ y_j = \prod_{i=1}^n A_i^{B_{ij}} \]
where $B = B_G$ is the signed adjacency matrix of the quiver associated to $G$ (see Section~\ref{subsec:quiver_from_graph}). 
\end{remark}

\begin{ex}
Let $G$ be the following graph on the left.
It is obtained from a integral polygon whose counterclockwise edge vectors $e_i$ are $(1,1)$, $(-1,1)$, $(-1,-1)$ and $(1,-1)$.
Let $Y_i$ be the counterclockwise loop around the face $i$, and $Z_1$ (resp. $Z_2$) be the loop in the middle (resp. right) picture.
\begin{center}
\includegraphics[scale=1]{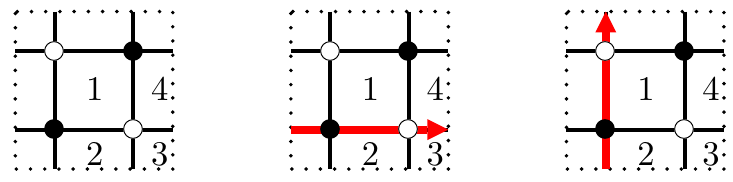}
\end{center}
Let $(A_1,A_2,A_3,A_4)\in(\C^*)^4$ be an arbitrary face weight on $G$.
It induces the following weights
\[ y_1 = \frac{A_4^2}{A_2^2},\quad y_2 = \frac{A_1^2}{A_3^2},\quad y_3 = \frac{A_2^2}{A_4^2},\quad y_4 = \frac{A_3^2}{A_1^2}, \quad z_1 = \frac{A_3 A_4}{A_1 A_2}, \quad z_2 = \frac{A_2 A_3}{A_1 A_4}. \]
They satisfy the following conditions
\[ y_1 y_3 = 1, \quad y_2 y_4 = 1, \quad z_1^2 = y_1 / y_2, \quad z_2^2 = 1/(y_1 y_2). \]
Since $y_1$ and $y_2$ are algebraically independent, the induced loop weight is a subspace of dimension $2$ inside $\L_G$ of dimension $5.$ ($L_G$ is of dimension $5$ because every loop weight in $\L_G$ satisfies $y_1y_2y_3y_4 = 1$.)
The map from face weight to loop weight is not injective.
We have the following 2-dimensional symmetries
\[ (A_1,A_2,A_3,A_4) \mapsto (\lambda A_1,\mu A_2,\lambda A_3,\mu A_4) \]
for $\lambda,\mu \in \C^*$.
\end{ex}


\subsection{Casimirs and Hamiltonians}\label{sec:hamiltonian_GK}

A \defemph{zig-zag path} on $\G$ is an oriented path on $\G$ which turns maximally left at white vertics and turns maximally right at black vertices \cite{Kenyon02, Postnikov06}.
They will always close up to form loops.
Notice that the projection of the zig-zag loops on the torus are in the same isotopy classes with the oriented loops $\a_i$ obtained from the primitive edges $e_i$ of $N$ (See Section~\ref{sec:bipartitegraph}).
So the oriented zig-zag loops are in bijection with $\{\a_i\}$.

The weight of these zig-zag loops, called \defemph{Casimirs}, generate the center of the Poisson algebra $\mathcal{O}(\L_\G)$ as described in the following proposition.

\begin{prop} [{\cite[Lemma 1.1]{GK13}}]
We consider the oriented zig-zag loops on a bipartite oriented surface graph $\G$.
Then as $Z$ runs over zig-zag loops, the functions $w_Z$ generate the center of the Poisson algebra $\mathcal{O}(\L_\G)$ of functions on $\L_\G$.
The product of all of them is 1.
This is the only relation between them.
\end{prop}

Recall the construction of $\G$ from a minimal admissible graph (an arrangement of $\{\a_i\}_{i=1}^n$) in Section~\ref{sec:bipartitegraph}.
We see that each edge $e$ of $\G$ has exactly two loops $\a_i,\a_j \in \{\a_1,\dots,\a_n\}$ crossing it.
Let $\a_i$ and $\a_j$ be as the following.
\begin{center}
\includegraphics[scale=1]{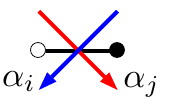}
\end{center}
Then we define a \defemph{reference perfect matching} $M_0$ to be the matching containing all edges $e$ such that $i>j$.
It is shown in \cite[Theorem 3.3]{GK13} that $M_0$ is indeed a perfect matching of $\G$.

We note that $M_0$ is not unique.
Since the indices of $\a_i$ can be read modulo $n$, another cyclic ordering gives another reference perfect matching.
There are also many other choices of $M_0$ including "fractional matchings", see \cite[Section 3.2-3.4]{GK13} for details.

%

\begin{remark}
We can show from \cite[Lemma 3.4]{GK13} that every quadrilateral face in $\G$ has exactly one side in $M_0$.
So every reference perfect matching $M_0$ constructed above always satisfies the requirement in Theorem~\ref{thm:hamiltonian_conserved}.
\end{remark}

Given a reference perfect matching $M_0$, the \defemph{weight} $w_{M_0}(M)$ of a perfect matching $M$ with respect to $M_0$ is defined to be the weight of the loop $[M]-[M_0]$, written in terms of $y_1,\dots,y_n,z_1,z_2$.

Recall the definition of $[M]_{M_0}$, the homology class of $M$ with respect to $M_0$, in Definition~\ref{def:hom_of_M}. 
The polygon with vertices at all homology classes $[M]_{M_0}\in H_1(\T,\Z)= \Z\times \Z$ coincides with the convex polygon $N$ up to translation \cite[Theorem~3.12]{GK13}.
Given a homology class $a\in H_1(\T,\Z)$, we let
\[ H_{M_0;a}:= \sum_{M} w_{M_0}(M) \]
where the sum runs over all perfect matchings $M$ of $\G$ having homology class $a$.
The \defemph{(modified) partition function} of perfect matchings of $\G$ is defined to be 
\[ P_{M_0} := \sum_a \sgn(a) H_{M_0;a} \]
where the sum runs over all possible homology classes of perfect matchings with respect to $M_0$.
The sign $\sgn(a)\in \{-1,1\}$ can be determined from a ``Kasteleyn matix", and they show up in the formula from the use of the determinant of a ``Kasterleyn operator", see \cite[Section 3.2]{GK13} for more details.

For $a\in H_1(\T,\Z) \cap \interior(N)$, a homology class which is an interior point of $N$, the function $H_{M_0;a}$ is called a \defemph{Hamiltonian}.
We note that a different choice of $M_0$ gives a different partition function and a different set of Hamiltonians.
However, they differ from each other by a common factor which lies in $\mathcal{O}(\L_\G)$.

These Hamiltonians are independent and commute under the Poisson bracket.

\begin{thm}[{\cite[Theorem 3.7]{GK13}}]
Let $\G$ be a minimal bipartite torus graph. Then
\begin{enumerate}
  \item The Hamiltonians $H_{M_0;a}$ commute under the Poisson bracket on $\L_\G$.
  \item The Hamiltonians are independent and their number is the half of the dimension of the generic symplectic leaf.
\end{enumerate}
\end{thm}

We also have that the partition function is invariant under the change of coordinates $i_{\G,\G'}$ (defined in Section~\ref{sec:manifold}).
This implies that all Hamiltonians are also invariant under the change of coordinates.
In fact, the map $i_{\G,\G'}$ is a unique rational transformation of face weights preserving the partition function, given a graph mutation from $\G$ to $\G'$. 

\begin{thm} [{ \cite[Theorem 4.7]{GK13} }] \label{thm:GK_urban_hamiltonian}
Given an urban renewal, there is a unique rational transformation of the weights preserving the partition function $P_a$.
This transformation is a Y-seed mutation.
\end{thm}

By counting the number of Hamiltonians and Casimirs, we can conclude on the integrability of the system.

\begin{thm}[{\cite[Theorem 1.2]{GK13}}] \label{thm:GK_integrable}
Let $M_0$ be a reference perfect matching obtained from a circular-order-preserving map.
The Hamiltonian flows of $H_{M_0,a}$ commute, providing an integrable system on $\X_N$.
Precisely, we get integrable systems on the generic symplectic leaves of $\X_N$, given by the level sets of the Casimirs.
\end{thm}

\begin{remark}
We notice that the integrable system described in this section is a classical dynamical system where the evolutions are Hamiltonian flows.
This system also contains a discrete dynamical system whose evolution is a change of coordinate $i_{G,G'}$ (Y-seed mutation on loop weights).

For a graph $\G$ periodic (up to a relabeling of vertices) under a sequence of urban renewals and shrinking of 2-valent vertices, we take the change of loop weight under such sequence to be the dynamic of a discrete system.
Since the graph is periodic, the dynamic is a Poisson map with respect to the Poisson bracket in Section~\ref{sec:manifold}.
The Casimirs and Hamiltonians return to the same form under such sequence.
Since they are also invariant by Theorem~\ref{thm:GK_urban_hamiltonian}, they are conserved quantities of the system.
By Theorem~\ref{thm:GK_integrable}, the quantities form a maximal set of Poisson-commuting invariants, hence the system is discrete Liouville integrable. 
\end{remark}

\section{Conclusion and discussion}

In this paper, we studied a discrete dynamic on a weighted bipartite torus graph, obtained from an urban renewal on the graph and cluster mutation on the weight.
The weight is defined differently from \cite{GK13}.
The graph can be any bipartite graph on a torus, not necessarily obtained from an integral polygon.
The Hamiltonians are defined and proved to be invariant under the mutation.

For a Q-system of type $A$, we constructed a weighted bipartite torus graph which is periodic under a sequence of mutations up to a face relabeling.
The weight changes according to the Q-system relation.
So the Hamiltonians are conserved quantities of the system.
This coincides with the result from \cite{GP16}.
We also showed that the Hamiltonians can be written as hard-particle partition functions on a certain graph, which coincides with the result in \cite{DFK10}.
A nondegenerate Poisson bracket is defined, and the Hamiltonians Poisson commute. 

For a Q-system of type $B$, a bipartite torus graph is constructed.
The graph is a double-cover of the dual graph associated with the Q-system quiver.
It is periodic under a sequence of mutations up to a face relabeling.
The weight are transformed according to the Q-system relation.
The Hamiltonians are conserved quantities of the system.
They can also be interpreted as hard-particle partition functions on a certain graph.
We conjecture that the conserved quantities Poisson commute under the Poisson bracket defined in Section~\ref{sec:bracket_b}.

One could wonder what happens for types $C, D$ and other exceptional types. The sequence of mutations in \cite[Theorem 3.1]{DFK09} contains a mutation at a vertex of degree greater than 4.
Recall that an urban renewal corresponds to a mutation at a vertex of degree 4 which has exactly two incoming and two outgoing arrows.
Although some mutations in type $A$ and $B$ happen at a vertex of degree less than 4, we fixed this issue by adding a frozen vertex.
However this technique is not applicable when the degree of a mutating vertex is greater than 4.
This problem might possibly be solved by unfolding or factoring the quiver so as to restore the 4-valent property.
We leave this for future work.


\bibliography{refs}
\bibliographystyle{alpha}

\end{document}